\documentclass[11pt,reqno]{amsart}

\usepackage{amsmath,amsthm}
\usepackage{amssymb,amsfonts}
\usepackage{latexsym}
\usepackage{mathrsfs}
\usepackage{bm}
\usepackage{tikz}
\usetikzlibrary{calc,intersections}
\usepackage[T1]{fontenc}

\DeclareMathAlphabet{\mathcal}{OMS}{cmsy}{m}{n}

\theoremstyle{plain}
\newtheorem{theorem}{Theorem}[section]
\newtheorem{lemma}[theorem]{Lemma}
\newtheorem{proposition}[theorem]{Proposition}
\newtheorem{corollary}[theorem]{Corollary}

\newtheorem*{maintheorem}{Theorem \ref{thm:main_theorem_I}}
\newtheorem*{maintheoremII}{Theorem \ref{thm:main_theorem_II}}

\newenvironment{theorembis}[1]
  {\renewcommand{\thetheorem}{\ref{#1}$'$}%
   \addtocounter{theorem}{-1}%
   \begin{theorem}}
  {\end{theorem}}

\newtheorem*{claim}{Claim}

\theoremstyle{definition}
\newtheorem{definition}[theorem]{Definition}

\newtheorem{notion}[theorem]{Notion Convention}

\theoremstyle{remark}
\newtheorem{remark}[theorem]{Remark}

\numberwithin{equation}{section}

\allowdisplaybreaks

\def\XXint#1#2#3{{\setbox0=\hbox{$#1{#2#3}{\int}$}
    \vcenter{\hbox{$#2#3$}}\kern-.5\wd0}}

\makeatletter
\def\@citestyle{\m@th\upshape\mdseries}
\def\citeform#1{{\color{blue}\bfseries#1}}
\def\@cite#1#2{{%
  \@citestyle[\citeform{#1}\if@tempswa, #2\fi]}}
\@ifundefined{cite }{%
  \expandafter\let\csname cite \endcsname\cite
  \edef\cite{\@nx\protect\@xp\@nx\csname cite \endcsname}%
}{}
\makeatother

\makeatletter 
   \newdimen\mex
   \def\niv{\mathrel{\hbox{\hglue .5\mex
       \vrule \@height 1.5\mex \@width .15\mex
       \vrule \@height .15\mex \@width 1.5\mex
       \hglue .5\mex}}}
   \def\vin{\mathrel{\hbox{\hglue .1\mex
       \vrule \@height .1\mex \@width 1.3\mex
       \vrule \@height 1.3\mex \@width .1\mex
       \hglue .4\mex}}}
\makeatother%

\renewcommand{\leq}{\leqslant}
\renewcommand{\geq}{\geqslant}

\newcommand{\inner}[2]{\langle #1\,,#2\rangle}

\newcommand{\C}{\mathbb{C}}
\newcommand{\R}{\mathbb{R}}

\newcommand{\B}{\mathbb{B}}
\newcommand{\U}{\mathbb{U}}
\renewcommand{\H}{\mathbb{H}}

\newcommand{\T}{\mathbf{T}}
\newcommand{\X}{\mathbf{X}}

\renewcommand{\O}{\mathcal{O}}

\DeclareMathOperator{\area}{area}%
\DeclareMathOperator{\dist}{dist}%
\DeclareMathOperator{\length}{length}%
\DeclareMathOperator{\spt}{spt}%
\DeclareMathOperator{\CH}{\mathsf{CH}}%
\DeclareMathOperator{\Mob}{\mathsf{M{\ddot{o}}b}}%
\DeclareMathOperator{\PSL}{\mathsf{PSL}}%
\DeclareMathOperator{\SO}{\mathsf{SO}}%

\newcommand{\Bcal}{\mathbf{B}}
\newcommand{\CC}{\mathcal{C}}

\newcommand{\Ascr}{\mathscr{A}}

\newcommand{\Dscr}{\mathscr{D}}
\newcommand{\Nscr}{\mathscr{N}}

\renewcommand{\v}{\bm{v}}

\makeatletter
\newcommand*\rel@kern[1]{\kern#1\dimexpr\macc@kerna}
\newcommand*\widebar[1]{%
  \begingroup
  \def\mathaccent##1##2{%
    \rel@kern{0.8}%
    \overline{\rel@kern{-0.8}\macc@nucleus\rel@kern{0.2}}%
    \rel@kern{-0.2}%
  }%
  \macc@depth\@ne
  \let\math@bgroup\@empty \let\math@egroup\macc@set@skewchar
  \mathsurround\z@ \frozen@everymath{\mathgroup\macc@group\relax}%
  \macc@set@skewchar\relax
  \let\mathaccentV\macc@nested@a
  \macc@nested@a\relax111{#1}%
  \endgroup
}
\makeatother
\newcommand{\Xbar}{\widebar{X}}

\begin{document}

\title{Asymptotic Plateau problem for two contours}

\author{Biao Wang}
\date{\today}

\address{Department of Mathematics and Computer Science\\
         The City University of New York, QCC\\
         222-05 56th Avenue Bayside, NY 11364\\}
\email{biwang@qcc.cuny.edu}

\begin{abstract}
   Let $\Gamma_{1}$ and $\Gamma_{2}$ be two disjoint rectifiable star-shaped
   Jordan curves in the asymptotic boundary $\partial_{\infty}\H^{3}$
   of the hyperbolic space $\H^3$. If the distance between $\Gamma_{1}$ and $\Gamma_{2}$
   are bounded above by a constant, then there exists an area minimizing annulus
   $\Pi\subset\H^3$, which is asymptotic to $\Gamma_{1}\cup\Gamma_{2}$.
   The main results of this paper are Theorem \ref{thm:main_theorem_I} and
   Theorem \ref{thm:main_theorem_II}.
\end{abstract}

\maketitle


\section{Introduction}\label{sec:introduction}

In this paper we study the asymptotic Plateau problem in hyperbolic $3$-space
$\H^3$ when the prescribed boundary data consists of
two disjoint Jordan curves at infinity.
There are several models for $\H^3$, among which we shall use
the Poincar{\'e} ball model and the upper half-space model.

The \emph{Poincar{\'e} ball model} of $\H^3$ is the open unit ball
\begin{equation}\label{eq:ball_model}
   \B^{3}=\{(u,v,w)\in\R^{3}\ | \ u^{2}+v^{2}+w^{2}<{}1\}
\end{equation}
equipped with the hyperbolic metric
$ds^{2}=4(du^{2}+dv^{2}+dw^{2})/(1-r^{2})^{2}$,
where $r=\sqrt{u^{2}+v^{2}+w^{2}}$.
The orientation preserving isometry group of $\B^3$
is denoted by $\Mob(\B^3)$, which consists of M\"obius transformations
that preserve the unit ball (see \cite[Theorem 1.7]{MT98}).
The hyperbolic $3$-space $\B^{3}$ has a natural compactification:
$\overline{\B^{3}}=\B^{3}\cup{}S_{\infty}^{2}$,
where $S_{\infty}^{2}\cong\C\cup\{\infty\}$ is called the
\emph{asymptotic boundary of} $\B^{3}$ or the \emph{idea boundary of}
$\B^{3}$ \emph{at infinity}.
Suppose that $X$ is a subset of $\B^{3}$, we define the
\emph{asymptotic boundary} of $X$ by
$\partial_{\infty}X=\Xbar \cap{}S_{\infty}^{2}$,
where $\Xbar$ is the closure of $X$ in $\overline{\B^{3}}$.
Obviously we have $\partial_\infty\B^3=S_{\infty}^2$.
If $P$ is a geodesic plane in $\B^{3}$, then $P$ is perpendicular to
$S_{\infty}^{2}$ and $C\stackrel{\text{def}}{=}\partial_{\infty}P$ is
an Euclidean round circle in $S_{\infty}^{2}$.
We also say that $P$ is {\em asymptotic to} $C$.

The \emph{upper half space model} of $\H^3$ is the upper half space
\begin{equation}\label{eq:upper_half_space}
   \U^{3}=\{z+tj\ |\ z\in\C\ \text{and}\ t>0\}
\end{equation}
equipped with the hyperbolic metric $ds^{2}=(|dz|^{2}+dt^{2})/t^{2}$,
where $z=x+iy$ for $x,y\in\R$.
The orientation preserving isometry group of $\U^3$ is denoted by
$\PSL_{2}(\C)$, which consists of linear fractional transformations.
It's well known that $\Mob(\B^3)\cong\PSL_{2}(\C)$.
The asymptotic boundary of $\U^3$ is
$\widehat\C=\C\cup\{\infty\}\cong{}S_{\infty}^{2}$.

For a collection of disjoint Jordan curves
$\Gamma=\{\Gamma_{1},\ldots,\Gamma_{k}\}$ in $S_{\infty}^{2}$, where $k\geq{}1$,
the \emph{asymptotic Plateau problem} in $\H^3$ asks the
existence of an (absolutely) area minimizing surfaces $\Sigma\subset\H^{3}$
asymptotic to $\Gamma$, that is,
$\partial_{\infty}\Sigma=\Gamma_{1}\cup\cdots\cup\Gamma_{k}$.
The asymptotic Plateau problem was first studied by Anderson
\cite{And82,And83} in hyperbolic $n$-space ($n\geq{}3$) for
arbitrary codimensions in the case when $k=1$.
In particular, using methods in geometric measure theory,
Anderson showed that there exists a complete embedded area minimizing plane
asymptotic to a given Jordan curve in $S_{\infty}^{2}$.

\begin{theorem}[{\cite[Theorem 4.1]{And83}}]\label{them:Theorem_of_Anderson}
For any Jordan curve $\Gamma$ in $S_{\infty}^{2}$, there exists a complete
embedded disk-type area minimizing surface $\Sigma\subset\H^3$,
which is asymptotic to $\Gamma$.
\end{theorem}

By the interior regularity results of geometric measure theory
(see \cite{Fed69}), the complete area minimizing disk $\Sigma$ in
Theorem \ref{them:Theorem_of_Anderson}
is smooth. Moreover if $\Sigma$ is an (\emph{absolutely}) area minimizing
surface asymptotic to a $C^{1,\alpha}$ Jordan curve $\Gamma$ in
$S_{\infty}^{2}$, then $\Sigma$ is $C^{1,\alpha}$ at infinity (see \cite{HL87});
the higher boundary regularity of $\Sigma$ at
infinity was studied by Lin in \cite{Lin89,Lin12}.
Moreover, the asymptotic behavior of
area-minimizing currents in hyperbolic space with higher codimensions
was studied by Lin in \cite{Lin89CPAM}.
The reader can read the survey \cite{Cos14} for other topics on asymptotic
Plateau problem.

In this paper we shall study the asymptotic Plateau problem in 
$\H^3$ when $\Gamma=\{\Gamma_{1},\ldots,\Gamma_{k}\}\subset{}S_{\infty}^{2}$
consists of \emph{two} components.

\begin{definition}
If $C_{1}$ and $C_{2}$ are two disjoint round circles in
$\partial_{\infty}\H^{3}=S_{\infty}^{2}$, we define the distance
between $C_{1}$ and $C_{2}$ as follows
\begin{equation}\label{eq:distance_circles}
   d(C_1,C_2)=\dist(P_1,P_2)\ ,
\end{equation}
where $\dist(\cdot,\cdot)$ is the hyperbolic distance of $\H^3$ and
$P_{i}$ is the totally geodesic plane asymptotic to $C_{i}$ for
$i=1,2$.
\end{definition}

\begin{remark}
In the remaining part of the paper, when we say \emph{circles}, we mean \emph{round}
circles; otherwise we shall say simple closed curves or Jordan curves.
\end{remark}

The following theorem of Gomes (see \cite[Proposition 3.2]{Gom87}
or \cite[Theorem 3.2]{Wan19}) partially solved the asymptotic Plateau problem
in $\H^3$ when two given disjoint Jordan curves in $S_{\infty}^{2}$ are circles.

\begin{theorem}[Gomes]\label{thm:Gomes1987-prop3.2-a}
Let $a_{c}\approx{}0.49577$ be the {\rm(}unique{\rm)} critical number
of the function $\varrho$ defined by \eqref{eq:Gomes function I}.
For two disjoint circles $C_{1},C_{2}\subset{}S_{\infty}^{2}$, if
\begin{equation}\label{eq:maximim_distance}
   d(C_{1},C_{2})\leq{}2\varrho(a_{c})\approx{}1.00229\ ,
\end{equation}
then there exists a minimal surface of revolution, that is, a
spherical catenoid in $\H^3$, which is asymptotic to $C_{1}\cup{}C_{2}$.
\end{theorem}

The above theorem of Gomes can't determine whether the catenoid is
area minimizing. Actually we even don't know if it is globally stable.
According to $\S$\ref{subsec:catenoids} any spherical catenoid $\CC$ in $\B^3$
can be determined uniquely up to isometry by the distance from $\CC$ to its
rotation axis.
Let $\CC_{a}$ denote the spherical catenoid in $\B^3$ which has distance
$a$ from itself to its rotation axis (see
Definition \ref{def:sphereical catenoid with parameter}).
The following theorem can determine the stability of spherical catenoids
according to the distances from catenoids to their rotation axes
(see \cite[Proposition 4.10]{BSE10} or \cite[Theorem 1.2]{Wan19}).

\begin{theorem}[B{\'e}rard and Sa Earp]\label{thm:BSE}
Let $a_{c}\approx{}0.49577$ be the {\rm(}unique{\rm)} critical number
of the function $\varrho$ defined by \eqref{eq:Gomes function I}.
\begin{enumerate}
   \item $\CC_{a}$ is unstable if $0<a<a_{c}$, and
   \item $\CC_{a}$ is globally stable if $a\geq{}a_{c}$.
\end{enumerate}
\end{theorem}

The following theorem is an equivalent form of Theorem \ref{thm:BSE}.

\begin{theorembis}{thm:BSE}\label{thm:Berard_Sa_Earp}
Let $C_{1}$ and $C_{2}$ be disjoint round circles in $S_{\infty}^{2}$.
\begin{enumerate}
  \item If $d(C_{1},C_{2})=2\varrho(a_{c})$, there exists exactly one
        globally stable catenoid in $\H^3$ asymptotic to $C_{1}\cup{}C_{2}$.
  \item If $0<d(C_{1},C_{2})<2\varrho(a_{c})$, there exist two
        catenoids in $\H^3$ asymptotic to $C_{1}\cup{}C_{2}$ such that one is unstable
        and the other one is globally stable.
\end{enumerate}
\end{theorembis}

We shall explain why there are two spherical catenoids in $\H^3$ asymptotic to the
disjoint round circles $C_{1}$ and $C_{2}$ in $S_{\infty}^{2}$
if $d(C_{1},C_{2})<2\varrho(a_{c})$.
Suppose that we have the family of catenoids $\{\CC_{a}\}_{a>0}$
in $\B^3$ such that their rotation axes are the same $u$-axis and they
are all symmetric about the $vw$-plane (see \eqref{eq:ball_model}).
According to the arguments in \cite[pp. 357--359]{Wan19},
the function $\varrho(a)$ is increasing on $(0,a_{c})$ and
decreasing on $(a_{c},\infty)$, and achieves its maximum value at $a=a_{c}$
(see Figure \ref{fig:Gomes_function}).
If $d(C_{1},C_{2})<2\varrho(a_{c})$, there exist exactly two positive constants
$a'<a_{c}<a''$ such that $d(C_{1},C_{2})=2\varrho(a')=2\varrho(a'')$,
that is to say, both $\CC_{a'}$ and $\CC_{a''}$ are asymptotic to
$C_{1}$ and $C_{2}$.

Recall that a stable minimal surface is locally area minimizing, so it could be
area minimizing. In the above theorem of B{\'e}rard and Sa Earp, we still don't know
whether a stable catenoid is area minimizing.
The following theorem (see also \cite[Theorem 1.5]{Wan19}) can determine some
area minimizing catenoids among the stable ones.
As in \cite[(4.2) and (4.3)]{Wan19}, we define
\begin{equation}\label{eq:least_area_constant}
   a_{l}=\cosh^{-1}\left(1/(1-K)\right)\approx{} 1.10055\ ,
\end{equation}
where the constant $K$ is defined by
\begin{equation}\label{eq:constant_K}
   K=\int_{0}^{1}\frac{1}{x^2}\left(\frac{1}{\sqrt{1-x^4}}-1\right)dx
     \approx{}0.40093\ .
\end{equation}

\begin{theorem}[Wang]\label{thm:least_area_catenoid}
Each catenoid $\CC_{a}$ is area minimizing if $a\geq{}a_{l}$.
\end{theorem}

According to the proof of above theorem in \cite{Wan19}, it seems that
there are still some area minimizing catenoids $\CC_a$ when $a_{c}\leq{}a<a_{l}$.
The following theorem is our first main result, which can determine all of the
(absolutely) area minimizing spherical catenoids among the globally stable ones.

\begin{theorem}\label{thm:main_theorem_I}
Let $a_{L}\approx{}0.847486$ be a constant given by
Theorem \ref{thm:area_difference}.
\begin{enumerate}
  \item If $a_{c}\leq{}a<a_{L}$, then each spherical catenoid $\CC_{a}$ is globally
        stable but not {\rm(}absolutely{\rm)} area minimizing.
  \item If $a\geq{}a_{L}$, then each spherical catenoid $\CC_{a}$ is {\rm(}absolutely{\rm)} area
        minimizing among all surfaces asymptotic to $\partial_{\infty}\CC_{a}$.
\end{enumerate}

Moreover, any {\rm(}absolutely{\rm)} area minimizing surface 
asymptotic to two disjoint round circles in $S_{\infty}^{2}$ is
one of the elements in $\{\CC_{a}\}_{a\geq{}a_{L}}$ up to isometry.
\end{theorem}

The following theorem is an equivalent form of Theorem \ref{thm:main_theorem_I}.

\begin{theorembis}{thm:main_theorem_I}\label{thm:main_theorem_I_prime}
Let $C_{1}$ and $C_{2}$ be disjoint round circles in $S_{\infty}^{2}$.
\begin{enumerate}
  \item If $d(C_{1},C_{2})=2\varrho(a_{c})$, there exists exactly one
        globally stable spherical catenoid, which is not {\rm(}absolutely{\rm)} area minimizing.
  \item If $2\varrho(a_{L})<d(C_{1},C_{2})<2\varrho(a_{c})$, there exist two
        spherical catenoids such that one is unstable and the other one is globally stable.
        But the stable catenoid is not {\rm(}absolutely{\rm)} area minimizing.
  \item If $0<d(C_{1},C_{2})\leq{}2\varrho(a_{L})$, there exist two spherical
        catenoids such that one is unstable and the other one is
        {\rm(}absolutely{\rm)} area minimizing.
\end{enumerate}

Moreover, if there is any {\rm(}absolutely{\rm)} area minimizing surface in
$\H^3$ asymptotic to $C_{1}\cup{}C_{2}$ in $S_{\infty}^{2}$, then it's a spherical catenoid and
the distance \eqref{eq:distance_circles} between $C_{1}$ and $C_{2}$ is
$\leq{}2\varrho(a_{L})$.
\end{theorembis}

Applying the results in Theorem \ref{thm:main_theorem_I} or
Theorem \ref{thm:main_theorem_I_prime}, we can say that we have solved the special
case of asymptotic Plateau problem when the prescribed boundary data consists of
two disjoint round circles.

Next we will generalize the above results to the case when the prescribed
boundary data consists of two disjoint rectifiable star-shaped Jordan curves.

\begin{definition}[Distance between two Jordan curves]
\label{def:distance_Jordan_curves}
Let $\Gamma_{1}$ and $\Gamma_{2}$ be two disjoint Jordan curves in
$S_{\infty}^{2}$. For $i=1,2$, let $\Delta_{i}$ be the disk component
of $S_{\infty}^{2}\setminus(\Gamma_{1}\cup\Gamma_{2})$ bounded by
$\Gamma_{i}$, and let $C_{i}$ be any round circle contained in $\Delta_{i}$.
The distance between $\Gamma_{1}$ and $\Gamma_{2}$ is defined as follows:
\begin{equation}\label{eq:diatance_between_Jordan_curves}
   d(\Gamma_{1},\Gamma_{2})=
   \inf\{d(C_{1},C_{2})\ |\ C_{i}\subset\Delta_{i}\ \text{for}\
   i=1,2\}\ ,
\end{equation}
where $d(C_{1},C_{2})$ is given by \eqref{eq:distance_circles} for
two disjoint circles $C_{1}$ and $C_{2}$.
\end{definition}

\begin{remark}In \cite[p.430]{dCGT86}, the authors
also defined the distance between Jordan curves in $S_{\infty}^{2}$,
which is different from \eqref{eq:diatance_between_Jordan_curves}.
\end{remark}

A set $\Omega$ in the plane $\R^2$ is called a
\emph{star-shaped domain} if there exists
a point $x_0$ in $\Omega$ such that for each point $x$ in $\Omega$ the line segment
from $x_0$ to $x$ is contained in $\Omega$.
The point $x_0$ is called a \emph{center} of the domain $\Omega$.
A star-shaped domain in the plane may have more than one center. In particular,
any interior point of a convex domain in the plane is its center.

\begin{definition}[Star-shaped Jordan curve]\label{def:star_shaped}
A Jordan curve $\Gamma\subset{}S_{\infty}^{2}$ is called
\emph{star-shaped} if two components $\Omega_{\pm}$ of
$S_{\infty}^{2}\setminus\Gamma$ are all star-shaped domains.
More precisely, there exist points $p_{+}\in\Omega_{+}$ and
$p_{-}\in\Omega_{-}$, a geodesic $\ell\subset\H^3$ connecting
$p_{+}$ and $p_{-}$ (i.e., $\partial_{\infty}\ell=\{p_{+},p_{-}\}$), and
an isometry $\phi:\H^{3}\to\U^{3}$ such that the following conditions
are satisfied:
\begin{enumerate}
  \item $\phi(\ell)$ is the $t$-axis of the upper-half space $\U^{3}$
        (see \eqref{eq:upper_half_space}) with $\phi(p_{+})=0$ and
        $\phi(p_{-})=\infty$.
  \item $\phi(\Omega_{+})\subset\widehat\C$ is a star-shaped domain whose
        center is the origin, and $\phi(\Omega_{-})\subset\widehat\C$ is a
        star-shaped domain whose center is at infinity, that is to say, 
        for any point $x\in\phi(\Omega_{-})$, the portion of the ray passing 
        through $x$ from the origin that starts at $x$ is contained 
        in the domain $\phi(\Omega_{-})$.
\end{enumerate}
The geodesic $\ell$ is called an \emph{axis} of $\Gamma$.
\end{definition}

A Jordan curve $\Gamma\subset{}S_{\infty}^{2}$ is \emph{rectifiable}
if its length is finite with respect to the spherical metric on $S_{\infty}^{2}$.
Our second main result can be stated as follows:

\begin{theorem}\label{thm:main_theorem_II}
Let $\Gamma_{1}$ and $\Gamma_{2}$ be disjoint rectifiable star-shaped
Jordan curves in $S_{\infty}^{2}$. If the distance between $\Gamma_{1}$ and
$\Gamma_{2}$ is bounded from above as follows
\begin{equation}\label{eq:upper_bound}
   d(\Gamma_{1},\Gamma_{2})<{}2\varrho(a_{L})\approx{}0.876895\ ,
\end{equation}
where $\varrho$ is the function defined by \eqref{eq:Gomes function I}
and $a_{L}\approx{}0.847486$ is the constant given by
Theorem \ref{thm:area_difference},
then there exists an embedded annulus-type area minimizing surface
$\Pi\subset\H^3$, which is asymptotic to $\Gamma_{1}\cup\Gamma_{2}$.

Moreover the upper bound in \eqref{eq:upper_bound} is optimal in the following
sense: If there is an area minimizing surface in $\H^3$
asymptotic to two disjoint round circles in $S_{\infty}^{2}$, then the distance
\eqref{eq:distance_circles} between the circles is $\leq{}2\varrho(a_{L})$.
\end{theorem}

\begin{remark}
The results of boundary regularity in \cite{HL87,Lin89,Lin12} also work for
Theorem \ref{thm:main_theorem_II}.
\end{remark}



\begin{remark}
Coskunuzer proved the following result (see Step 1
in the proof of the Key Lemma in \cite{Cos09}):
Let $\Gamma$ be a Jordan curve in $S_{\infty}^{2}$ with at least
one $C^{1}$-smooth point. If $\Gamma^{\pm}\subset{}S_{\infty}^{2}$
are two Jordan curves in opposite sides of $\Gamma$ and sufficiently close
to $\Gamma$, then there exists an area minimizing annulus
asymptotic to $\Gamma^{+}\cup\Gamma^{-}$.
\end{remark}

\subsection*{Organization of the paper}
In $\S$\ref{sec:prelim}, we review some definitions on minimal surfaces and catenoids.
In $\S$\ref{sec:main_theorem_I}, we shall prove Theorem \ref{thm:area_difference} at first,
which is crucial for proving Theorem \ref{thm:main_theorem_I} in the same section.
In $\S$\ref{sec:main_II}, we shall prove several results before we prove Theorem \ref{thm:main_theorem_II} in the same section.

\subsection*{Acknowledgement}
This project is partially supported by PSC-CUNY Research Award \#61073-0049.
The author also thanks Professor Lin Fang-Hua for confirming that
the key result of Theorem 2.2 in \cite{HL87} works for both homotopically and
homologically area minimizing surfaces in hyperbolic space.

\section{Preliminaries}\label{sec:prelim}

In this section, we shall review the definitions and basic properties of minimal surfaces
in hyperbolic $3$-space $\H^3$.

\subsection{Review of minimal surfaces}

Suppose that $\Sigma$ is a surface immersed
in a complete Riemannian $3$-manifold
$M^3$. We pick up a local orthonormal frame field
$\{e_1,e_2,e_3\}$ for $M^3$ such that, restricted to $\Sigma$,
the vectors $\{e_1,e_2\}$ are tangent to $\Sigma$ and the
vector $e_3$ is perpendicular to $\Sigma$. Let
$A=(h_{ij})_{2\times{}2}$ be the second fundamental
form of $\Sigma$, whose entries $h_{ij}$ are defined by
$h_{ij}=\inner{\nabla_{e_i}e_{3}}{e_{j}}$ for $i,j=1,2$,
where $\nabla$ is the covariant derivative in $M^3$, and
$\inner{\cdot}{\cdot}$ is the metric of $M^3$.
The immersed surface $\Sigma$ is called a \emph{minimal surface} in $M^3$
if its \emph{mean curvature} $H=h_{11}+h_{22}$ is identically equal
to zero.

\begin{definition}[Area minimizing disk]
A compact disk-type minimal surface $\Sigma$ in $\H^{3}$
is called an \emph{area minimizing surface} in $\H^{3}$
if $\Sigma$ has least area among the compact disks in $\H^{3}$ which are
homotopic to $\Sigma$ rel $\partial{}\Sigma$.

A noncompact complete disk-type surface $\Sigma$ is called
an \emph{area minimizing surface} in $\H^{3}$
if any compact disk-type sub-domain of $\Sigma$ is an area minimizing
surface in $\H^{3}$.
\end{definition}

\begin{definition}[Area minimizing annulus]\label{def:least_area_annulus}
Let $S$ be a compact annulus-type minimal surface immersed in $\H^3$,
whose boundary consists of two disjoint Jordan curves $C_1,C_2$, and
let $D_1,D_2$ be two area minimizing disks spanning $C_1,C_2$ respectively.
The annulus $S$ is called an \emph{area minimizing surface} in $\H^3$ if
\begin{enumerate}
  \item $\area(S)<\area(D_{1})+\area(D_{2})$, and
  \item $\area(S)\leq\area(S')$ for each annulus $S'$ homotopic
        to $S$ rel $\partial{}S$,
\end{enumerate}
where $\area(\cdot)$ denotes the area of the surfaces in $\H^3$.

A noncompact complete minimal annulus $\Pi\subset\H^3$ whose asymptotic
boundary consists of the union of two disjoint Jordan curves
$\Gamma_{1}$ and $\Gamma_{2}$ in $S_{\infty}^{2}$ is called
an \emph{area minimizing surface}
if any compact annulus-type subdomain of $\Pi$, which is
homotopically equivalent to $\Pi$, is a compact area minimizing annulus.
\end{definition}

\begin{remark}Condition (1) in Definition \ref{def:least_area_annulus} is
necessary (see \cite{Dou31,AS79,MY1982(t)}). See also the proof of
Theorem \ref{thm:main_theorem_I}.
\end{remark}

\begin{definition}[Absolutely area minimizing surface]
\label{def:absolutely_area_minimizing}
Suppose $S\subset\H^{3}$ is a compact surface
with boundary. The surface $S$ is called \emph{absolutely area minimizing}
if $S$ has least area among all compact surfaces with the same boundary,
where these surfaces are not necessarily homotopoic to $S$ rel $\partial{}S$.

A noncompact complete surface $\Sigma$ in $\H^3$ is called an
\emph{absolutely area minimizing surface} provided
each compact portion of it is absolutely area minimizing.
\end{definition}


\begin{notion}
In the literature, an area minimizing surface is also called a
\emph{homotopically area minimizing surface}, and
an absolutely area minimizing surface is also called a
\emph{homologically area minimizing surface}.
\end{notion}

\subsection{Spherical catenoids in $\H^3$}\label{subsec:catenoids}
In this subsection, we shall review some very basic properties of
catenoids defined in the Poincar{\'e} ball model of $\H^3$.
See also \cite{Mor81,Hsi82,dCD83,Gom87,BSE09,BSE10,Wan19}.

Let $G\cong\SO(2)$ be a subgroup of $\Mob(\B^{3})$ that leaves
a geodesic $\gamma_{0}\subset\B^{3}$ pointwise fixed. We call $G$ the
\emph{spherical group} of $\B^{3}$ and $\gamma$ the
\emph{rotation axis} of $G$.
For two round circles $C_{1}$ and $C_{2}$ in $\B^{3}$, if there is a
geodesic $\gamma_{0}$,
such that both $C_{1}$ and $C_{2}$ are invariant under the spherical
group that fixes $\gamma_{0}$ pointwise, then $C_{1}$
and $C_{2}$ are said to be \emph{coaxial}, and $\gamma_{0}$ is
called the {\em rotation axis} of $C_{1}$ and $C_{2}$.
It's well known that any two disjoint round circles $C_{1}$ and $C_{2}$ in
$S_{\infty}^{2}$ are always coaxial (see \cite{Wan19}).

Suppose that $G$ is the spherical group of $\B^{3}$ associated with
the geodesic
\begin{equation}\label{eq:rotation axis}
   \gamma_{0}=\{(u,0,0)\in\B^3\ |\ -1<u<1\}\ ,
\end{equation}
then we have $\B^3/G\cong\B_{+}^2$, where
$\B_{+}^2:=\{(u,v)\in\B^{2}\subset\B^{3}\ |\ v\geq{}0\}$
is still considered as a subset of $\B^3$.
For any point $p=(u,v)\in\B_{+}^2$, there is a unique
geodesic segment $\gamma'$ through $p$ which is
perpendicular to $\gamma_{0}$ at $q$.
Set $x=\dist(O,q)$ and $y=\dist(p,q)=\dist(p,\gamma_{0})$
(see Figure~\ref{fig:intrinsic metric}).
It's well known that $\B_{+}^2$ can be equipped with
the \emph{metric of warped product} in terms
of the parameters $x$ and $y$, that is, $ds^2=\cosh^{2}y\cdot{}dx^2+dy^2$,
where $dx$ represents the hyperbolic metric on the geodesic
$\gamma_0$ defined in \eqref{eq:rotation axis}.

If $\CC$ is a minimal surface of revolution in $\B^3$ with
respect to the axis $\gamma_{0}$, where $\gamma_{0}$ is defined by
\eqref{eq:rotation axis}, then it is called a \emph{catenoid}
and the curve $\sigma=\CC\cap\B_{+}^{2}$ is called the
\emph{generating curve} or a \emph{catenary} of $\CC$.
Let $\sigma\subset\B_{+}^{2}$ be the generating curve of a
minimal catnoid $\CC$. Suppose that the parametric equations
of $\sigma$ are given by:
$x=x(s)$ and $y=y(s)$, where $s\in(-\infty,\infty)$ is an
arc length parameter of $\sigma$. By the arguments in
\cite[pp. 486--488]{Hsi82}, the curve $\sigma$ satisfies the
following equations
\begin{equation}\label{eq:differential equations of Pi}
   \frac{2\pi\sinh{}y\cdot\cosh^{2}y}
   {\sqrt{\cosh^2{}y+(y')^2}}=
   2\pi\sinh{}y\cdot\cosh{}y\cdot\sin\theta=k\
   (\text{constant})\ ,
\end{equation}
where $y'=dy/dx$ and $\theta$ is the angle between the tangent
vector of $\sigma$ and the vector $e_y=\partial/\partial{}y$
at the point $(x(s),y(s))$ (see Figure~\ref{fig:diff eq of sigma}).

\begin{figure}[htbp]
\begin{center}
  \begin{minipage}[tb]{0.5\textwidth}
  \centering
  \includegraphics[scale=0.8]{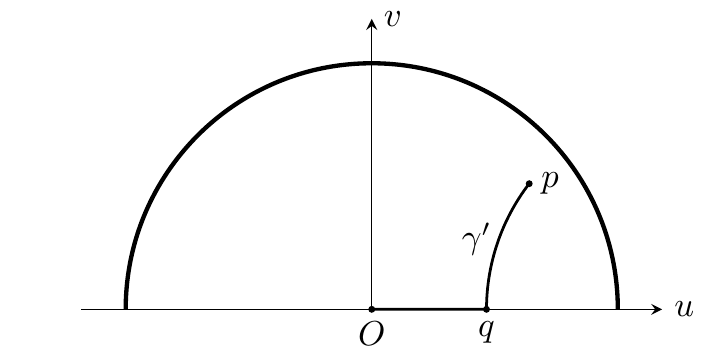}
  \caption{The new coordinates: $x=\dist(O,q)$ and $y=\dist(p,q)$.}
  \label{fig:intrinsic metric}
  \end{minipage}%
  \begin{minipage}[tb]{0.5\textwidth}
  \centering
  \includegraphics[scale=0.7]{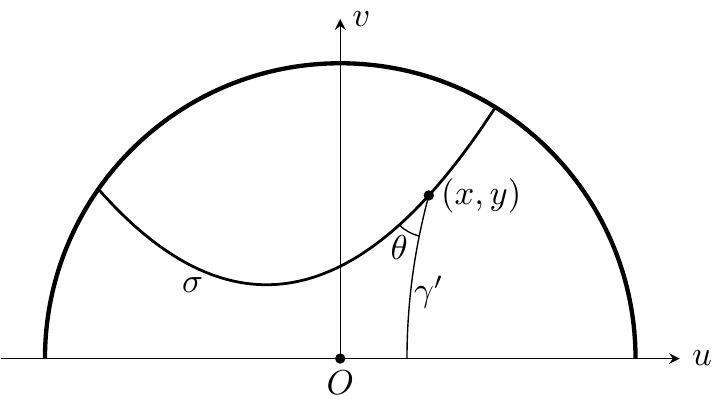}
  \caption{$\theta$ is the angle between $\sigma$ and $\gamma'$.}
  \label{fig:diff eq of sigma}
  \end{minipage}
\end{center}
\end{figure}

By the arguments in \cite[pp.54--58]{Gom87}), up to isometry,
we can assume that the curve $\sigma$ is only symmetric about the $v$-axis
and intersects the $v$-axis orthogonally at $y_0=y(0)$,
and so $y'(0)=0$.
Now we solve for $dx/dy$ in terms of $y$
in \eqref{eq:differential equations of Pi} and integrate
$dx/dy$ from $y_0$ to $y$ for any $y\geq{}y_0$ (see \cite{Hsi82} or \cite{Wan19}),
then we have the following equality
\begin{equation}\label{eq: catenary equation}
   x(y)=\int_{y_0}^{y}\frac{\sinh(2y_0)}{\cosh{}t}
             \frac{dt}{\sqrt{\sinh^2(2t)-\sinh^2(2y_0)}}\ .
\end{equation}
The limit of the function $x(y)$ defined by \eqref{eq: catenary equation}
is finite as $y\to\infty$ for any fixed $y_{0}>0$.
Replacing the initial data $y_0$ by a parameter $a\in(0,\infty)$ in
\eqref{eq: catenary equation}, we can define a function $\varrho(a)$ of
the parameter $a$ as follows (see Figure \ref{fig:Gomes_function}):
\begin{equation}\label{eq:Gomes function I}
   \varrho(a)=\int_{a}^{\infty}\frac{\sinh(2a)}{\cosh{}t}
              \frac{dt}{\sqrt{\sinh^2(2t)-\sinh^2(2a)}}\ .
\end{equation}
The function \eqref{eq:Gomes function I} has a unique critical value
$a_{c}\approx{}0.49577$ so that $\varrho(a)<\varrho(a_{c})$ for all
$a\in(0,a_{c})\cup{}(a_{c},\infty)$ by \cite[Lemma 3.3]{Wan19}.

Let $\sigma_a\subset\B_{+}^2$ be the catenary defined by 
\eqref{eq: catenary equation} for $y_{0}=a$,
which is symmetric about the $v$-axis and whose initial
data is $a$ (actually the hyperbolic distance between $\sigma_{a}$
and the origin of $\B_{+}^2$ is equal to $a$).

\begin{definition}\label{def:sphereical catenoid with parameter}
For $0<a<\infty$, the surface of revolution around the axis $\gamma_{0}$ in
\eqref{eq:rotation axis} generated by the catenary $\sigma_a$
is called a \emph{catenoid}, which is denoted by $\CC_{a}$.
\end{definition}


\begin{notion}\label{notion_2}
In this paper, let $\mathscr{S}$ denote the set of the spherical catenoids
in $\B^3$ that have the \emph{same} rotation axis $\gamma_{0}$ (see \eqref{eq:rotation axis})
and the \emph{same} symmetric plane
$P_{0}=\{(u,v,w)\in\R^3\ |\ v^{2}+w^{2}<1\ \text{and}\ u=0\}$.
Let
\begin{equation}\label{eq:foliated_region}
   \T_{c}:=\cup\{\CC_{a}\in\mathscr{S}\,:\,a\geq{}a_{c}\}
\end{equation}
be a subregion of $\B^3$.
The region $\T_{c}$ is foliated by the spherical
catenoids $\CC_{a}$ in $\mathscr{S}$ for all $a\geq{}a_{c}$
according to the three dimensional version of
the first statement in \cite[Proposition 4.8]{BSE10}.
\end{notion}

\begin{figure}[htbp]
  \begin{center}
     \includegraphics[scale=1]{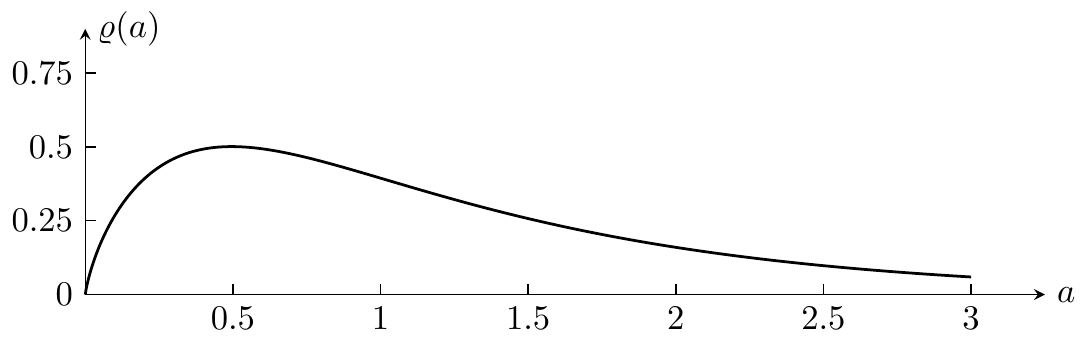}
  \end{center}
  \caption{The graph of the function $\varrho(a)$ defined by
  \eqref{eq:Gomes function I} for
  $a\in[0,3]$. This function $\varrho(a)$ has a unique critical number.}
  \label{fig:Gomes_function}
\end{figure}

\section{Existence of area minimizing spherical catenoids in $\H^3$}
\label{sec:main_theorem_I}

In this section we shall prove Theorem \ref{thm:main_theorem_I} in
$\S$\ref{subsec:main_theorem_I}.

\subsection{Area difference between catenoids and geodesic planes}
Let $\CC_{a}\in\mathscr{S}$ be a spherical catenoid in $\B^3$, whose rotation
axis is $\gamma_0$. For any $r\geq{}a$, let
\begin{equation}\label{eq:Sigma_a_r}
   \Sigma_{a,r}=\CC_{a}\cap\Nscr_{r}(\gamma_{0})\ ,
\end{equation}
where $\Nscr_{r}(\gamma_{0})$ denotes the $r$-neighborhood of $\gamma_{0}$
in $\B^3$. Then $\Sigma_{a,r}$ is a compact surface of revolution whose boundary
consists of two round circles $C^{\pm}$, which are invariant under the
spherical group of the rotation axis $\gamma_{0}$.
The area of $\Sigma_{a,r}$ can be calculated by coarea formula (see \cite[(4.8)]{Wan19})
\begin{equation}\label{eq:compact_annulus_area}
   \area(\Sigma_{a,r})=
      \int_{a}^{r}\left(4\pi\sinh{}t\cdot
      \frac{\sinh(2t)}{\sqrt{\sinh^2(2t)-\sinh^2(2a)}}\right)dt\ .
\end{equation}

Let $\Delta_{r}^{\pm}\subset\B^3$ be the totally geodesic disks bounded by $C^{\pm}$
respectively, then the area of $\Delta_{r}^{\pm}$ is given by
$\area(\Delta_{r}^{+})=\area(\Delta_{r}^{-})=2\pi(\cosh{}r-1)$
(see \cite[Theorem 7.2.2]{Bea95})
since the radii of $\Delta_{r}^{\pm}$ are both equal to $r$.

\begin{lemma}Let
$\Phi(a,r)=\area(\Sigma_{a,r})-(\area(\Delta_{r}^{+})+\area(\Delta_{r}^{-}))$
be the area difference, then $\Phi(\cdot,r)$ is increasing and bounded for
$a\leq{}r<\infty$.
\end{lemma}

\begin{proof}Actually, the function $\Phi$ is given by
\begin{equation*}
\begin{aligned}
   \Phi(a,r)=&\,4\pi\int_{a}^{r}\frac{\sinh{}t\cdot\sinh(2t)}
              {\sqrt{\sinh^2(2t)-\sinh^2(2a)}}\,dt-4\pi(\cosh{}r-1)\\
            =&\,4\pi\int_{a}^{r}\sinh{}t\cdot
                \left(\frac{\sinh(2t)}{\sqrt{\sinh^2(2t)-\sinh^2(2a)}}-1\right)dt\\
             &\,-4\pi(\cosh{}a-1)\ ,
\end{aligned}
\end{equation*}
for $r\geq{}a$. The first integral term in the second equality is positive for all
$r\geq{}a$, so $\Phi(\cdot,r)$ is an increasing function on $r$ for each fixed $a>0$.

For any fixed $a>0$, the function $\Phi(a,r)$ is bounded for any $r\geq{}a$.
First of all, $\Phi(a,r)$ is bounded from below, since
$\Phi(a,r)\geq{}-4\pi(\cosh{}a-1)$ for $r\geq{}a$.
On the other hand, by Lemma 4.1 and (4.8) in \cite{Wan19}, we have
\begin{align*}
   \Phi(a,r)
      \leq&\,4\pi\int_{a}^{\infty}\sinh{}t\cdot
              \left(\frac{\sinh(2t)}{\sqrt{\sinh^2(2t)-\sinh^2(2a)}}-1\right)dt\\
          &\,-4\pi(\cosh{}a-1) \\
      \leq&\,4\pi{}K\cosh{}a-4\pi(\cosh{}a-1)\ ,
\end{align*}
for any $r\geq{}a$, where $K$ is defined by \eqref{eq:constant_K}.
\end{proof}

\begin{figure}[htbp]
  \begin{center}
     \includegraphics[scale=0.9]{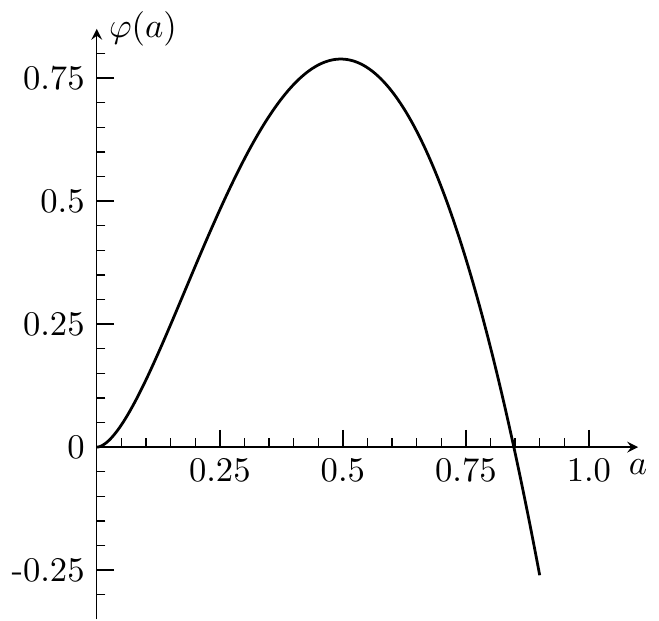}
  \end{center}
  \vspace{-0.5cm}
  \caption{The graph of the function $\varphi(a)$ defined by
  \eqref{eq:area_difference} for $a\in[0,0.9]$.}
  \label{fig:area_difference_function}
\end{figure}

Therefore, for any fixed $a\geq{}0$, the the limit
of $\Phi(a,r)$ as $r\to\infty$ is well defined. Let
$\varphi(a)=\lim\limits_{r\to\infty}\Phi(a,r)$, then $\varphi(a)$
is a function of $a$, which can be written as follows
(see Figure \ref{fig:area_difference_function})
\begin{equation}\label{eq:area_difference}
\begin{split}
   \varphi(a)=&\,4\pi\int_{a}^{\infty}\sinh{}t\cdot
                  \left(\frac{\sinh(2t)}{\sqrt{\sinh^2(2t)-\sinh^2(2a)}}-1\right)dt\\
              &\,-4\pi(\cosh{}a-1)\ .
\end{split}
\end{equation}


\begin{theorem}\label{thm:area_difference}
The function $\varphi(a)$ defined by \eqref{eq:area_difference}, $a\geq{}0$,
has a unique zero $a_{L}>a_{c}>0$ such that the following
statements are true.
\begin{enumerate}
  \item If $0<a<a_{L}$, then $\varphi(a)>0$.
  \item If $a>a_{L}$, then $\varphi(a)<0$ .
\end{enumerate}
\end{theorem}

\begin{remark}
By numerical computation, $a_{L}\approx{}0.847486$.
It seems that the critical number of the function $\varphi(a)$ is also equal to $a_c$.
\end{remark}

\begin{lemma}\label{lem:MVT}
Let
\begin{equation*}
   f(x)=-30\cosh(3x)-18\cosh(5x)+10\sinh(7x)+15(1-K)\cosh(8x)
\end{equation*}
be a function defined for $x\geq{}0$,
where $K\approx{}0.40093$ is defined by \eqref{eq:constant_K}.
Then $f(x)$ is increasing on $[0,\infty)$, and has
exactly one zero $a_{0}\in(0,a_{c})$.
\end{lemma}

\begin{proof}It's easy to verify that the following functions
\begin{equation*}
\begin{aligned}
   f_{1}(x)&=-14\cosh(5x)+10\sinh(7x) \\
   f_{2}(x)&=-30\cosh(3x)-4\cosh(5x)+15(1-K)\cosh(8x)
\end{aligned}
\end{equation*}
are increasing for $x\geq{}0$, so is $f(x)=f_{1}(x)+f_{2}(x)$ for $x\geq{}0$.
Actually direct computation shows that
$f_{1}'(x)=-70\sinh(5x)+70\cosh(7x)>0$ for all $x\geq{}0$,
where we use the inequalities $\cosh(7x)\geq{}\cosh(5x)>\sinh(5x)$ for all $x\geq{}0$,
so $f_{1}(x)$ is increasing on $[0,\infty)$. Direct computation shows that
\begin{equation*}
\begin{aligned}
   f_{2}''(x)&=-270\cosh(3x)-100\cosh(5x)+960(1-K)\cosh(8x)\\
             &>-270\cosh(3x)-100\cosh(5x)+480\cosh(8x)
\end{aligned}
\end{equation*}
is positive for all $x\geq{}0$, since $\cosh(8x)>\cosh(5x)>\cosh(3x)$
for all $x>0$.
It's easy to verify that $f_{2}'(0)=0$,
so $f_{2}'(x)>0$ for all $x>0$, which means that
$f_{2}(x)$ is increasing on $[0,\infty)$.

Since $f(0)=-48+15(1-K)<0$ and
\begin{equation*}
  f(\log(3/2))=\frac{171374697-215561285\cdot{}K}{1119744}>0\ ,
\end{equation*}
there exists a unique zero $0<a_{0}<\log(3/2)\approx{}0.405465<a_{c}$ of $f(x)$.
\end{proof}

\begin{proof}[{\bf Proof of Theorem \ref{thm:area_difference}}]
We shall prove the theorem in four steps.
Suppose that the catenoids in the family $\mathscr{S}=\{\CC_{a}\}_{a>0}$
have the same rotation axis $\gamma_{0}$
and the same symmetric plane $P_{0}$ (see Notion Convention \ref{notion_2}).

\vskip 0.25cm

\noindent\textbf{Step 1}. $\varphi(a)>0$ \emph{for} $0<a<a_{c}$. \emph{In particular,
we have} $\varphi(a_{c})\geq{}0$.

\begin{proof}[{\bf Proof of Step 1}]
Otherwise, if $\varphi(a)\leq{}0$ for any fixed $a\in(0,a_{c})$,
then we have $\Phi(a,r)<0$ for any $r>a$, that is,
\begin{equation}\label{eq:condition_least_area_annulus}
   \area(\Sigma_{a,r})<\area(\Delta_{r}^{+})+\area(\Delta_{r}^{-})\ ,
   \quad r>a\ ,
\end{equation}
where $\Sigma_{a,r}=\CC_{a}\cap\Nscr_{r}(\gamma_{0})$ and
$\Delta_{r}^{\pm}$ are the totally geodesic disks bounded by
the round circles $\partial\Sigma_{a,r}$.
We may choose $r\gg{}a$ such that
$\partial\Sigma_{a,r}\subset\T_{c}$, since $\partial_{\infty}\CC_{a}$
is contained in $\partial_{\infty}\T_{c}$ by Theorem \ref{thm:Gomes1987-prop3.2-a},
where $\T_{c}$ is defined by \eqref{eq:foliated_region}.
By Theorem \ref{thm:BSE}, $\CC_{a}$ is unstable since $a<a_{c}$.
According to the three dimensional version of the second statement of
Proposition 4.8 in \cite{BSE10}, $\Sigma_{a,r}$ is also unstable.

Let $\Omega\subset\B^{3}$ be the region
bounded by $\Delta_{r}^{+}$, $\Delta_{r}^{-}$ and $\Sigma_{a,r}$.
Since $\Omega$ is a simply connected region whose boundary
$\partial\Omega=\Delta_{r}^{+}\cup\Delta_{r}^{-}\cup\Sigma_{a,r}$
is mean convex with respect to the inward normal vector field, together
with the condition \eqref{eq:condition_least_area_annulus},
there exists an annulus-type area minimizing surface $\Sigma'\subset\Omega$ with
$\partial\Sigma'=\partial\Sigma_{a,r}$
according to \cite{AS79,MY1982(mz)}.

But the existence of $\Sigma'$ is impossible. The argument is as follows.
Since $\Sigma_{a,r}$ is unstable, $\Sigma'$ is not identical to $\Sigma_{a,r}$.
On the other hand, similar to
the arguments in \cite[$\S$4]{Wan19}, we know that $\Sigma'$ is a minimal
surface of revolution about $\gamma_{0}$, and it's also symmetric about the same
plane $P_0$, therefore it's a portion of
some catenoid $\CC_{a'}\in\mathscr{S}$, where $a'=\dist(\Sigma',\gamma_{0})$.
Since $\Sigma'\subset\Omega$, we have
\begin{equation*}
   a'=\dist(\Sigma',\gamma_{0})<\dist(\Sigma_{a,r},\gamma_{0})=a<a_{c}\ .
\end{equation*}
This implies that $\CC_{a'}$ is unstable. Recall that
$\partial\Sigma'=\partial\Sigma_{a,r}\subset\T_{c}$,
the compact minimal annulus $\Sigma'$ is unstable
according to the second statement of Proposition 4.8 in \cite{BSE10},
therefore it can \emph{not} be area minimizing either.

This is a contradiction to the above assumption, which implies
$\varphi(a)$ must be positive for all $0<a<a_{c}$.
In particular, this implies that $\varphi(a_{c})\geq{}0$.
\end{proof}

\noindent\textbf{Step 2}. $\varphi(a)<0$ \emph{if} $a\geq{}a_{l}>a_{c}$,
\emph{where} $a_{l}\approx{} 1.10055$
\emph{is defined by \eqref{eq:least_area_constant}}.
\emph{In particular}, $\varphi(a)\to{}-\infty$ \emph{as} $a\to\infty$.

\begin{proof}[{\bf Proof of Step 2}]
Using the substitution $t\mapsto{}t+a$,
we have the following estimate
according to the arguments in the proof of Lemma 4.1 in \cite{Wan19}
\begin{equation*}
\begin{aligned}
   \varphi(a)<4\pi{}K\cosh{}a-4\pi(\cosh{}a-1)
             =-4\pi(1-K)\cosh{}a+4\pi
\end{aligned}
\end{equation*}
for all $a\geq{}0$,
where $K<1$ is defined by \eqref{eq:constant_K}.
By Lemma 4.1 and (4.8) in \cite{Wan19}, we have $\varphi(a)<0$
if $a\geq{}a_{l}$. Furthermore, $\varphi(a)>-4\pi(\cosh{}a-1)$ for all $a>0$,
therefore we have $\varphi(a)\to{}-\infty$ as $a\to\infty$ by the squeeze theorem.
\end{proof}


\noindent\textbf{Step 3}. $\varphi(a)$ \emph{is concave downward if} $a>a_{0}$,
\emph{where} $a_{0}<a_{c}$ \emph{is the constant determined in Lemma \ref{lem:MVT}}.

\begin{proof}[{\bf Proof of Step 3}]
Direct computation shows that the second derivative of
$\varphi(a)$ can be written as $\varphi''(a)=I_{1}(a)+I_{2}(a)$, where
\begin{align*}
   I_{1}(a)=&\,4\pi\int_{0}^{\infty}\sinh(a+t)\cdot
                \left(\frac{\sinh(2a+2t)}{\sqrt{\sinh^2(2a+2t)-\sinh^2(2a)}}-1\right)dt \\
            &\,-4\pi{}K\cosh{}a
\end{align*}
and
\begin{align*}
   I_{2}(a)=&\,\scalebox{0.95}{$-4\pi$}\int_{0}^{\infty}
                \frac{\scalebox{0.9}{$5\cosh(a+t)-3\cosh(3a+3t)-3\cosh(5a+t)+\cosh(7a+3t)$}}
               {\scalebox{0.95}{$\sqrt{\sinh^2(2a+2t)-\sinh^2(2a)}\,\cdot\,\sinh^{2}(4a+2t)$}}dt\\
            &\,\scalebox{0.95}{$-4\pi{}(1-K)\cosh{}a$}\ .
\end{align*}
Similar to the argument in Step 2, we have $I_{1}(a)<0$ for any $a\geq{}0$.
Next we need show that $I_{2}(a)<0$ for $a>a_{0}$. We shall estimate both the numerator
and denominator of the integrand of $I_{2}(a)$:
For the numerator, we have 
\begin{equation*}
\begin{aligned}
   \text{numerator}
     &=\scalebox{0.95}{$5\cosh(a+t)-3\cosh(3a+3t)-3\cosh(5a+t)+\cosh(7a+3t)$}\\
     &\geq{}\scalebox{0.95}{$-3\cosh(3a)e^{3t}-3\cosh(5a)e^{t}+\sinh(7a)e^{3t}$}\ .
\end{aligned}
\end{equation*}
For the denominator, we have
\begin{align*}
   \text{denominator}
    &=\sqrt{\sinh^2(2a+2t)-\sinh^2(2a)}\cdot\sinh^{2}(4a+2t)\\
    &\leq\sinh(2a+2t)\cdot(\sinh(4a+2t))^{2}\\
    &\leq\cosh(2a)\cosh^{2}(4a)e^{6t}\ .
\end{align*}
So we have the following estimates
\begin{align*}
   I_{2}(a)
      \leq&\,-4\pi\int_{0}^{\infty}
             \frac{-3\cosh(3a)e^{-3t}
             -3\cosh(5a)e^{-5t}+\sinh(7a)e^{-3t}}{\cosh(2a)\cosh^{2}(4a)}dt\\
          &\,-4\pi{}(1-K)\cosh{}a\\
      =&\,-4\pi\left\{\frac{-\cosh(3a)-\frac{3}{5}\cosh(5a)+\frac{1}{3}\sinh(7a)}
          {\cosh(2a)\cosh^{2}(4a)}\right\}\\
       &\,-4\pi{}(1-K)\cosh{}a\\
      \leq&\,-\frac{\scalebox{0.9}{$4\pi$}}{\scalebox{0.9}{$30$}}\cdot
      \frac{\scalebox{0.9}{$-30\cosh(3a)-18\cosh(5a)+10\sinh(7a)+15(1-K)\cosh(8a)$}}
           {\scalebox{0.9}{$\cosh(2a)\cosh^{2}(4a)$}}\ ,
\end{align*}
where we use the facts that $\cosh(2a)\geq{}\cosh(a)\geq{}1$ and
$\cosh^{2}(4a)\geq\cosh(8a)/2$ for the last inequality.
According to Lemma \ref{lem:MVT}, $I_{2}(a)\leq{}0$ if $a>a_{0}$.

Therefore $\varphi''(a)<0$ as long as $a>a_{0}$.
In other words, $\varphi(a)$ is concave
downward on $(a_{0},\infty)$.
\end{proof}


\noindent\textbf{Step 4}. $\varphi(a_{c})>0$.

\begin{proof}[{\bf Proof of Step 4}]
By Lemma 3.3 in \cite{Wan19}, we know that
$0<a_{c}<A_{3}$, where the constant $A_{3}\approx{}0.530638$ is defined
by \cite[(6.1)]{Wan19}.
According to Step 3, $\varphi(a)$ is concave downward on $(a_{0},A_{3})\ni{}a_{c}$,
since $\varphi(a_{0})>0$ by Step 1 and $\varphi(A_{3})\approx{}0.781314>0$ by direct
numerical computation using softwares, we have $\varphi(a_{c})>0$.
\end{proof}

According to Steps 1--3 there exists a (unique) zero $a_{L}>a_{0}>0$ of
$\varphi(a)$ such that the conditions in Theorem \ref{thm:area_difference} are satisfied.
Moreover, we know that $a_{L}>a_{c}$ since $\varphi(a)>0$ if $a\in(0,a_{c}]$
according to Step 4.
\end{proof}


\subsection{Area minimizing annuli asymptotic to circles}
In order to prove the last statement of Theorem \ref{thm:main_theorem_I},
we have to determine whether a connected minimal surface $\Sigma\subset\H^3$
asymptotic to two disjoint round circles in $S_{\infty}^2$ is a spherical catenoid.
The following theorem of Levitt and Rosenberg shows that
$\Sigma$ is a spherical catenoid if it is regular at infinity
(see \cite[Theorem 3.2]{LR85} or \cite[Theorem 3]{dCGT86}).
Without the regularity at infinity, $\Sigma$ still could be a spherical catenoid
as long as it is (absolutely) area minimizing (see Corollary \ref{cor:LR}).

\begin{definition}
For an integer $k\geq{}1$, a complete minimal surface $\Sigma$ of
$\H^3$ is $C^{k}$-\emph{regular at infinity} if $\partial_\infty\Sigma$
is a $C^k$-submanifold of $S_{\infty}^2$ and
$\overline\Sigma=\Sigma\cup\partial_\infty\Sigma$ is a
$C^k$-submanifold (with boundary) of $\overline{\H^3}$.
\end{definition}

\begin{theorem}[Levitt and Rosenberg]\label{thm:LR85}
Let $\CC$ be a connected minimal surface immersed in $\H^3$ whose asymptotic
boundary consists of two disjoint round circles $C_{1}$ and $C_{2}$
in $S_{\infty}^{2}$. If $\CC$ is $C^2$-regular at infinity, then $\CC$ is a
spherical catenoid asymptotic to $C_{1}\cup{}C_{2}$.
\end{theorem}

The following corollary is a direct application of Theorem \ref{thm:LR85} and
boundary regularity, which will be applied to prove the last statement
of Theorem \ref{thm:main_theorem_I}.

\begin{corollary}\label{cor:LR}
Let $\CC$ be a connected minimal surface immersed in $\H^3$ whose asymptotic
boundary consists of two disjoint round circles $C_{1}$ and $C_{2}$ in $S_{\infty}^{2}$.
If $\CC$ is an {\rm(}absolutely{\rm)} area minimizing surface,
then $\CC$ is a spherical catenoid asymptotic to $C_{1}\cup{}C_{2}$.
\end{corollary}

\begin{proof}
Since round circles in $S_{\infty}^{2}$ are always smooth, the (absolutely)
area minimizing surface $\CC$ is $C^{2}$-regular at infinity by the results
of boundary regularity in \cite{HL87,Lin89,Lin12}
\footnote{The key result of Theorem 2.2 in \cite{HL87} works for
both homotopically and homologically area minimizing surfaces in $\U^3$.}.
Therefore $\CC$ must be a minimal surface of revolution by the theorem of
Levitt and Rosenberg, that is, $\Sigma$ is a spherical catenoid, whose asymptotic
boundary is $C_{1}\cup{}C_{2}$.
\end{proof}


\subsection{Existence of area minimizing catenoids}\label{subsec:main_theorem_I}
Now we are able to prove the the first main result of this paper.

\begin{maintheorem}
There exists a constant $a_{L}\approx{}0.847486$ given by
Theorem \ref{thm:area_difference}.
\begin{enumerate}
  \item If $a_{c}\leq{}a<a_{L}$, then each spherical catenoid $\CC_{a}$ is globally
        stable but not {\rm(}absolutely{\rm)} area minimizing.
  \item If $a\geq{}a_{L}$, then each spherical catenoid $\CC_{a}$ is
        {\rm(}absolutely{\rm)} area minimizing among all surfaces asymptotic to
        $\partial_{\infty}\CC_{a}$.
\end{enumerate}

Moreover, any {\rm(}absolutely{\rm)} area minimizing surface asymptotic to
two disjoint round circles in $S_{\infty}^{2}$ must be one element in
$\{\CC_{a}\}_{a\geq{}a_{L}}$ up to isometry.
\end{maintheorem}

\begin{proof}
Suppose that all the spherical catenoids in the family
$\mathscr{S}=\{\CC_{a}\}_{a>0}$ have
the same axis $\gamma_{0}$ and the same symmetric plane $P_0$
(see Notion Convention \ref{notion_2}).

(1) If $a_{c}\leq{}a<a_{L}$, then $\varphi(a)>0$ by Theorem \ref{thm:area_difference},
so there exists some $r\gg{}a$ such that $\Phi(a,r)>0$, this
is equivalent to the following inequality
\begin{equation}
   \area(\Sigma_{a,r})>\area(\Delta_{r}^{+})+\area(\Delta_{r}^{-})
   =4\pi(\cosh{}r-1)\ ,
\end{equation}
where $\Sigma_{a,r}$ is defined by \eqref{eq:Sigma_a_r}, and
$\Delta_{r}^{\pm}$ are the totally geodesic disks bounded by
the two boundary components of $\Sigma_{a,r}$ respectively.

Next we shall construct annuli $\Pi_{a,r}(s)$ such that
$\partial\Pi_{a,r}(s)=\partial\Sigma_{a,r}$ for all $0<s\ll{}r$, and
the area of $\Pi_{a,r}(s)$ is less than that of $\Sigma_{a,r}$
when $s>0$ is sufficiently small.
Therefore $\CC_{a}$ is not an  area minimizing surface if
$a_{c}\leq{}a<a_{L}$, and it is not absolutely area minimizing either.

The distance between $\Delta_{r}^{+}$ and $\Delta_{r}^{-}$
is given by (see \cite[(3.11)]{Wan19})
\begin{equation}\label{eq:distance_disks}
   L=\dist(\Delta_{r}^{+},\Delta_{r}^{-})=
   2\int_{a}^{r}\frac{\sinh(2a)}{\cosh{}t}
   \frac{dt}{\sqrt{\sinh^2(2t)-\sinh^2(2a)}}\ .
\end{equation}
Let $0<s\ll{}r$ be sufficiently small, and let $Y_{a,r}(s)$ be the annulus-type
region of $\partial\Nscr_{s}(\gamma_{0})$ between $\Delta_{r}^{+}$
and $\Delta_{r}^{-}$, where $\Nscr_{s}(\gamma_{0})$ is the $s$-neighborhood
of $\gamma_{0}$, then $Y_{a,r}(s)$ is an equidistant cylinder with radius
$s$ and height $L$, so its area is easy to obtain
\begin{equation}
   \area(Y_{a,r}(s))=2\pi{}L\sinh{}s\cosh{}s\ .
\end{equation}
Let $D_{r}^{\pm}(s)=\Delta_{r}^{\pm}\cap\Nscr_{s}(\gamma_{0})$, then
$D_{r}^{\pm}(s)$ are totally geodesic disks with radii $s$.
We can define a new annulus $\Pi_{a,r}(s)$ as follows:
\begin{equation}\label{new_annulus_type_surface}
   \Pi_{a,r}(s)=Y_{a,r}(s)\cup(\overline{\Delta_{r}^{+}\setminus{}D_{r}^{+}(s)})
   \cup(\overline{\Delta_{r}^{-}\setminus{}D_{r}^{-}(s)})\ .
\end{equation}
Obviously $\partial\Pi_{a,r}(s)=\partial\Sigma_{a,r}$ and
$\Pi_{a,r}(s)$ is homotopic to $\Sigma_{a,r}$
rel $\partial\Sigma_{a,r}$.
Recall that $\area(D_{r}^{+}(s))+\area(D_{r}^{-}(s))=4\pi(\cosh{}s-1)$, so
the area of $\Pi_{a,r}(s)$ is
\begin{equation}
   \area(\Pi_{a,r}(s))=2\pi{}L\sinh{}s\cosh{}s+
   4\pi(\cosh{}r-1)-4\pi(\cosh{}s-1)\ .
\end{equation}
Since $2\pi{}L\sinh{}s\cosh{}s-4\pi(\cosh{}s-1)\to{}0$ as $s\to{}0$,
we can make $s$ sufficiently small so that
$2\pi{}L\sinh{}s\cosh{}s-4\pi(\cosh{}s-1)<\area(\Sigma_{a,r})-4\pi(\cosh{}r-1)$,
and then $\area(\Pi_{a,r}(s))<\area(\Sigma_{a,r})$.

Therefore, if $a_{c}\leq{}a<a_{L}$, then each stable catenoid $\CC_{a}$ is not
(absolutely) area minimizing.

(2) On the other hand, $\varphi(a)\leq{}0$ for $a\geq{}a_{L}$
according to Theorem \ref{thm:area_difference}, hence for any $r\geq{}a$,
we have $\Phi(a,r)<0$, which implies the following inequality
\begin{equation*}
   \area(\Sigma_{a,r})<\area(\Delta_{r}^{+})+\area(\Delta_{r}^{-})\ ,
   \quad\text{for all}\ r>a\ .
\end{equation*}
Recall that $\Sigma_{a,r}\subset\CC_{a}\subset\T_{c}$ for all $r\geq{}a$,
since $a\geq{}a_{L}>a_{c}$,
where $\T_c$ is defined by \eqref{eq:foliated_region}.
We claim that $\Sigma_{a,r}$ is a compact area minimizing annulus for all $r>a$.
Otherwise, similar to the arguments in the first step of the proof of
Theorem \ref{thm:area_difference}, there exists a compact area minimizing annulus
$\Sigma'$ such that $\partial\Sigma'=\partial\Sigma_{a,r}$ and it's a portion
of some spherical catenoid $\CC_{a'}\in\mathscr{S}$,
where $a'$ is the distance from $\CC_{a'}$ to its
rotation axis. Since the elements in $\{\CC_{a}\}_{a\geq{}a_{c}}$ are disjoint
to each other, we have $a'<a_{c}$. Similar to the arguments in the proof of
Theorem \ref{thm:area_difference}, $\Sigma'$ can't be area minimizing.
This is a contradiction.
So $\Sigma_{a,r}$ is area minimizing for all $r>a$, which implies that
$\CC_{a}$ is an area minimizing surface as long as $a\geq{}a_{L}$.

Next we will show that $\CC_{a}$ is actually an absolutely area
minimizing surface if $a\geq{}a_{L}$.
Let $\Sigma$ be an absolutely area minimizing surface which is asymptotic to
$C_{1}\cup{}C_{2}=:\partial_{\infty}\CC_{a}$, where $a\geq{}a_{L}$.
According to Corollary \ref{cor:LR},
$\Sigma$ must be a spherical catenoid asymptotic to $C_{1}\cup{}C_{2}$.
Recall that there are exactly two spherical catenoids asymptotic to
$C_{1}\cup{}C_{2}$ (see \cite{BSE10} or \cite{Wan19}),
one is $\CC_{a}$ whereas the other one is unstable.
Since $\Sigma$ is also area minimizing, $\Sigma$ must be identical to $\CC_a$,
where $a\geq{}a_{L}$.

For the last statement, let $\Pi\subset\H^{3}$ be an (absolutely) area
minimizing surface asymptotic to two round circles $C_1$ and $C_2$
in $S_{\infty}^{2}$. 
Then $\Pi$ must be a spherical catenoid by Corollary \ref{cor:LR}.
Let $a$ be the distance from $\Pi$ to its rotation axis, then $a\geq{}a_{L}$ by the
above arguments.
\end{proof}

\section{Existence of complete area minimizing annuli in $\H^3$}\label{sec:main_II}

In this section we shall prove Theorem \ref{thm:main_theorem_II}
(see $\S$\ref{subsec:proof_of_main_theorem_II}). At first, we shall prove
three important results: Proposition \ref{prop:single_disk},
Theorem \ref{thm:finite_density_at_infinity} and
Proposition \ref{lem:catenoid_intersection}.
To finish the proof of Theorem \ref{thm:main_theorem_II}, we also need the
help of geometric measure theory, in particular the theory of varifolds,
see \cite{Fed69,All75,Sim83(book),LY02,KP08,Mor16} for details.

Let $\Lambda$ be a set in $S_{\infty}^{2}$, the \emph{convex hull} of $\Lambda$,
which is denoted by $\CH(\Lambda)$, is the intersection of all the closed half spaces
in $\H^3$ whose asymptotic boundary 
contains $\Lambda$.
Suppose that $\Gamma_{1}$ and $\Gamma_{2}$ are two \emph{disjoint}
(star-shaped) Jordan curves in $S_{\infty}^{2}$.
Obviously $\partial_{\infty}\CH(\Gamma_{1}\cup\Gamma_{2})=\Gamma_{1}\cup\Gamma_{2}$.
The boundary of $\CH(\Gamma_{1}\cup\Gamma_{2})$ consists of three
surfaces in $\B^{3}$:
\begin{itemize}
  \item $\Dscr_{1}$ and $\Dscr_{2}$ are disk-type surfaces asymptotic
        to $\Gamma_{1}$ and $\Gamma_{2}$ respectively,
  \item $\Ascr$ is an annulus-type surface asymptotic
        to $\Gamma_{1}\cup\Gamma_{2}$.
\end{itemize}
Moreover $\partial\CH(\Gamma_{1}\cup\Gamma_{2})=\Dscr_{1}\cup\Dscr_{2}\cup\Ascr$
is mean convex with respect to the inward normal vector field.
For any minimal surface $\Sigma_{i}$ asymptotic to $\Gamma_{i}$, $i=1,2$,
and any minimal surface $\Pi$ asymptotic to $\Gamma_{1}\cup\Gamma_{2}$, it's
well known that $\Sigma_{1},\ \Sigma_{2},\ \Pi\subset\CH(\Gamma_{1}\cup\Gamma_{2})$
according to \cite[Theorem 19.2]{Sim83(book)}.


\subsection{Minimal surfaces asymptotic to star-shaped curves}
\label{subsec:star_shaped_curve}

Let $\Gamma\subset{}S_{\infty}^{2}$ be a star-shaped Jordan curve.
According to \cite[Theorem 4.1]{And83} and
\cite[Theorem 4.1]{HL87} (see also \cite[Example on pp.14--15]{Has86}),
there exists a \emph{unique} complete embedded disk-type minimal
surface $\Sigma\subset\B^3$ asymptotic to $\Gamma$, which minimizes area
in the category of immersed surfaces (no topological restriction)
asymptotic to $\Gamma$. In other words, $\Sigma$ is actually an absolutely
area minimizing surface asymptotic to $\Gamma$.

\begin{definition}\label{def:minimal_surfaces_star_shaped_boundary}
Let $\Gamma\subset{}S_{\infty}^{2}$ be a star-shaped Jordan curve
with an axis $\ell$ and let $\Sigma\subset\B^3$ be the minimal disk
asymptotic to $\Gamma$.
Let $\phi:\B^{3}\to\U^{3}$ be the isometry
that maps the axis $\ell$ of $\Gamma$ to the $t$-axis of $\U^3$
(see \eqref{eq:upper_half_space}).
For any positive real number $\lambda$, we define
\begin{equation}\label{eq:hyperbolic_translation}
   h_{\lambda}=\phi^{-1}\circ{}m_{\lambda}\circ\phi\ ,
\end{equation}
where $m_{\lambda}(z,t)=(\lambda{}z,\lambda{}t)$ for any $(z,t)\in\U^3$.
Then each $h_{\lambda}$ is an isometry of $\B^3$ that
translates a point in the geodesic $\ell$ at distance $\log\lambda$ along $\ell$.

For any $\lambda>0$, the surface $\Sigma_{\lambda}=h_{\lambda}(\Sigma)$
is an area minimizing disk asymptotic to the Jordan
curve $\Gamma_{\lambda}=h_{\lambda}(\Gamma)$. In particular $\Sigma_{\lambda}=\Sigma$
when ${\lambda}=1$.
\end{definition}

By \cite[Theorem 4.1]{HL87} and \cite[Corollary 2.4]{Lin89}
(see also \cite[Example on pp.14--15]{Has86}),
we have the following corollary
(still using the above notations and settings in
Definition \ref{def:minimal_surfaces_star_shaped_boundary}).

\begin{proposition}\label{prop:star_shaped_Jordan_curve}
The area minimizing disk $\Sigma\subset\B^3$ asymptotic to a star-shaped Jordan curve
$\Gamma\subset{}S_{\infty}^{2}$ is a Killing graph
{\rm(}see Definition 10.4.1 in \cite{Lop13}{\rm)}.
Moreover the family of complete area minimizing disks
$\{\Sigma_{\lambda}\}_{\lambda>0}$ foliates $\B^3$.
\end{proposition}

Next we try to understand the intersection of a minimal surface
asymptotic to a star-shaped Jordan with a $3$-ball. We expect
that this intersection just consists of exactly one component
when the radius of the ball is sufficiently large.
More precisely we have the following proposition.

\begin{proposition}\label{prop:single_disk}
Let $\Sigma\subset\H^3$ be a minimal surface asymptotic to
a star-shaped Jordan curve $\Gamma\subset{}S_{\infty}^{2}$.
Let $B^{3}(p,r)$ be any $3$-ball in $\U^3$ with the center $p$ and the
radius $r$, where $p\in\H^3$ is an arbitrary point.
If $r$ is sufficiently large, then $B^{3}(p,r)\cap\Sigma$ consists of exactly
one disk, whose boundary is simple closed curve.
Moreover, $B^{3}(p,r)\cap\Sigma$ converges to $\Gamma$ as $r\to\infty$.
\end{proposition}

Before we prove Proposition \ref{prop:single_disk}, we need prove the following
lemma.

\begin{lemma}\label{lem:Jordan_curve_single}
Let $\Sigma\subset\U^3$ be a minimal surface asymptotic to
a star-shaped Jordan curve $\Gamma\subset\widehat{\C}$.
Let $P(t)$ be the horizontal plane through the point $(0,0,t)$ for
$t>0$. There exists a positive number $\rho_{\Gamma}$ such that
$P(t)\cap\Sigma$ consists of exactly one simple closed curve
for all $t\in[0,\rho_{\Gamma})$.
Moreover $P(t)\cap\Sigma$ converges to $\Gamma$ as $t\to{}0$.
\end{lemma}

\begin{proof}According to \cite{HL87} or \cite{Lin89CPAM}, there exists
a constant $\rho_{\Gamma}$ depending on $\Gamma$ such that
\begin{equation}\label{eq:noncompact_annulus}
   \Sigma':=(\Sigma\cup\Gamma)\cap\{(x,y,t)\in\overline{\U^3}\ |\ t<\rho_{\Gamma}\}
\end{equation}
is a finite union of surfaces with boundary which can be viewed as a graph
over $\Gamma\times[0,\rho_{\Gamma})$.
We assume that $\Gamma$ is a star shaped Jordan curve,
therefore $\Gamma\times[0,\rho_{\Gamma})$
is an annulus, so is $\Sigma'$. This means that
$P(t)\cap\Sigma$ consists of exactly one simple closed curve if
$t<\rho_\Gamma$. As $t\to{}0$, $P(t)\cap\Sigma$ converges to
$\Gamma$.
\end{proof}

\begin{proof}[\bf{Proof of Proposition \ref{prop:single_disk}}]
Consider the upper half space model $\U^3$.
When $r$ is sufficiently large (which might depend on the choice of
$p$), $B^{3}(p,r)$ is sufficiently close to a horizontal horosphere in $\U^3$.
In particular, when $r$ is sufficiently large, we have
$(\Sigma\cup\Gamma)\cap\left(\U^{3}\setminus\widebar{B^{3}(p,r)}\right)\subset\Sigma'$,
where $\Sigma'$ is given by \eqref{eq:noncompact_annulus}.
Applying Lemma \ref{lem:Jordan_curve_single}, we prove the statement of the
proposition.
\end{proof}

\subsection{Density at infinity}
Let $\Sigma\subset\H^{3}$ be a minimal surface asymptotic to a Jordan curve $
\Gamma$ in $S_{\infty}^{2}$. Fix a point $p\in\H^{3}$, for any $r>0$, let
\begin{equation}\label{eq:density}
   \Theta(\Sigma,p,r)=\frac{\area(\Sigma\cap{}B^{3}(p,r))}{4\pi\sinh^{2}(r/2)}
     =\frac{\area(\Sigma\cap{}B^{3}(p,r))}{2\pi(\cosh{}r-1)}\ ,
\end{equation}
where $B^{3}(p,r)\subset\H^3$ is an open three ball with (hyperbolic) radius
$r$ centered at $p$.
According the hyperbolic version of monotonicity formula
(see \cite[Theorem 1]{And82}), $\Theta(\Sigma,p,r)$ is a nondecreasing
function of $r>0$, so the limit of $\Theta(\Sigma,p,r)$ exists
as $r\to\infty$. Note that $B^{3}(p,r)\subset{}B^{3}(q,r+\dist(p,q))$
for any point $q\in\H^{3}$, from which it easily follows that
$\lim\limits_{r\to\infty}\Theta(\Sigma,p,r)$ is independent of the choice of $p$, and
therefore we call
\begin{equation}\label{eq:infinite_density}
   \Theta_{\infty}(\Sigma)=\lim_{r\to\infty}\Theta(\Sigma,p,r)
\end{equation}
the \emph{density of $\Sigma$ at infinity} (see \cite{Whi2016}).

The following result belongs to Gromov \cite[Theorem 8.3.A]{Gro83}
(see also \cite{EWW2002}), which is crucial to prove
Theorem \ref{thm:main_theorem_II}.

\begin{theorem}[Gromov]\label{thm:finite_density_at_infinity}
If $\Sigma\subset\H^3$ is a minimal surface asymptotic to
a rectifiable star-shaped Jordan curve $\Gamma\subset{}S_{\infty}^{2}$,
then $\Theta_{\infty}(\Sigma)$ is finite.
\end{theorem}

\begin{proof}
Let $\CH(\Gamma)\subset\H^{3}$ be the convex hull of $\Gamma$,
then $\Sigma$ is contained in $\CH(\Gamma)$.
Choose a point $p$ in $\Sigma$, and consider the geodesic cone $\mathscr{C}$
over $\Gamma$ with vertex $p$. Then $\mathscr{C}$ is also contained in $\CH(\Gamma)$.
Since $\Gamma$ is rectifiable, $\Theta(\mathscr{C},p,r)$ is a finite constant
for any $r>0$, so is $\Theta_{\infty}(\mathscr{C})$.

Let $B^{3}(p,r)\subset\H^{3}$ be the $3$-ball with the radius $r$ and the center $p$.
According to Proposition \ref{prop:single_disk},
there exists a constant $r_{0}>0$ such that
$\Sigma(r):=B^{3}(p,r)\cap\Sigma$ consists of exactly one disktype component
and $A(r):=\partial{}B^{3}(p,r)\cap\CH(\Gamma)$ is an annulus for all
$r>r_{0}$. Also set $\mathscr{C}(r):=B^{3}(p,r)\cap\mathscr{C}$.

Obviously both $\partial\Sigma(r)=\partial{}B^{3}(p,r)\cap\Sigma(r)$ and
$\partial\mathscr{C}(r)=\partial{}B^{3}(p,r)\cap\mathscr{C}$
are contained in $A(r)$ for all $r>r_{0}$. Let $E(r)\subset{}A(r)$ be the
domain bounded by $\partial\Sigma(r)$ and
$\partial\mathscr{C}(r)$.
Then $\Sigma(r)$ and $\mathscr{C}(r)\cup{}E(r)$ are the surfaces in $\H^3$
with the same boundary. Since $\Sigma(r)$ is an absolutely area minimizing
surface for any $r>r_0$, we must have
\begin{equation*}
   \area(\Sigma(r))<\area(\mathscr{C}(r)\cup{}E(r))<\area(\mathscr{C}(r))+
   \area(A(r))
\end{equation*}
for all $r>r_{0}$.
The area of $A(r)$ can be estimated as the product of
an exponentially small factor and the length of $\partial\mathscr{C}(r)$,
where the former term is obtained from the argument that is similar
to \cite[p.40]{Thu80} by the definition of convex hull
(see also \cite[p.111]{Gro83}). More precisely, we have
\begin{equation*}
   \area(A(r))=\O\big(e^{-(r-r_{0})}\length(\partial\mathscr{C}(r))\big)
              =\O\big(e^{-(r-r_{0})}\sinh{}r\big)
\end{equation*}
for all $r>r_0$, where the second equality comes from the fact that
$\Theta(\mathscr{C},p,r)$ is a finite constant for any $r>0$, which can
imply that $\length(\partial\mathscr{C}(r))/(2\pi\sinh{}r)$ is the
same constant for all $r>0$.

Therefore we have the following estimates
\begin{align*}
   \Theta(\Sigma,p,r)
      &=\frac{\area(\Sigma\cap{}B^{3}(p,r))}{4\pi\sinh^{2}(r/2)}\\
      &<\frac{\area(\mathscr{C}\cap{}B^{3}(p,r))}{{4\pi\sinh^{2}(r/2)}}+
        \frac{\area(A(r))}{{4\pi\sinh^{2}(r/2)}}\\
      &=\Theta_{\infty}(\mathscr{C})+
        \O\left(e^{-(r-r_{0})}\frac{\sinh{}r}{\sinh^{2}(r/2)}\right)\\
      &=\Theta_{\infty}(\mathscr{C})+
        \O\left(e^{-(r-r_{0})}\frac{\sinh{}r}{\cosh{}r-1}\right)
\end{align*}
for all $r>r_0$. As $r\to\infty$, we have
$\Theta_{\infty}(\Sigma)\leq\Theta_{\infty}(\mathscr{C})<\infty$.
\end{proof}



\subsection{Intersection of minimal surfaces}
In this subsection, we study the intersection of a spherical catenoid with a minimal disk
asymptotic to a star-shaped Jordan curve. At first let's fix some notations in the
following definition.


\begin{definition}\label{def:interior_and_exterior_regions}
A catenoid $\CC$ with $\partial_{\infty}\CC=C_{1}\cup{}C_{2}$
divides the hyperbolic space $\B^3$ into two regions, the one
containing the rotation axis of $\CC$ is homeomorphic to a solid cylinder,
and the other one is homeomorphic to a solid torus. The former is called
the \emph{interior region} of $\CC$, denoted by $\X$; the latter
is called the \emph{exterior region} of $\CC$, denoted by $\T$.
One can verify that $\partial\T=\CC=\partial\X$,
$\partial_{\infty}\X=B_{1}\cup{}B_{2}$ and
$\partial_{\infty}\T=S_{\infty}\setminus(B_{1}\cup{}B_{2})$,
where $B_{i}$ is the disk-type component of
$S_{\infty}^{2}\setminus(C_{1}\cup{}C_{2})$ bounded by $C_{i}$
for $i=1,2$. Note that both $\T$ and $\X$ are the subregions
of $\B^3$ with mean convex boundary.
\end{definition}

\begin{proposition}\label{lem:catenoid_intersection}
Let $\CC\subset\B^{3}$ be an area minimizing catenoid asymptotic to
$C_{1}\cup{}C_{2}$, whose interior region is denoted by $\X$.
Let $\Gamma\subset{}S^{2}_{\infty}$
be a star-shaped Jordan curve with the axis $\ell$
which separates $C_{1}$ and $C_{2}$ {\rm(}that is, two components of
$S_{\infty}^{2}\setminus\Gamma$ contain $C_{1}$ and $C_{2}$ respetively{\rm)}.
Suppose that $\Sigma\subset\B^3$ is the area minimizing disk asymptotic to
$\Gamma$, then $\Sigma$ must intersect $\CC$ transversely,
and $\Sigma\cap\X$ consists of a single disk-type subdomain
of $\Sigma$, denoted by $\Delta$, such that
\begin{enumerate}
  \item $\partial\Delta$ is a Jordan curve which is essential in $\CC$, and
  \item $(\Sigma\setminus\overline{\Delta})\cap(\X\cup\CC)=\emptyset$.
\end{enumerate}
\end{proposition}

\begin{proof}
According to \cite[Theorem 4.1]{And83} and\cite[Theorem 4.1]{HL87},
$\Sigma$ is the unique embedded minimal disk asymptotic to $\Gamma$, which has
the least area among all surfaces asymptotic to $\Gamma$.

It's easy to see that $\Gamma$ is essential in $\T=\B^{3}\setminus\X$,
which implies that the complete minimal disk $\Sigma$ must intersect $\CC$.

\vskip 0.25cm
\noindent\textbf{Claim 1}: \emph{If} $\Sigma$ \emph{and} $\CC$
\emph{intersect transversely at some simple closed curve} $\alpha$,
\emph{then} $\alpha$ \emph{must be essential in} $\CC$.

\begin{proof}[{\bf Proof of Claim 1}]
Otherwise, there exist two compact minimal disks $D\subset\CC$ and
$\Delta\subset\Sigma$ such that $\partial{}D=\alpha=\partial{}\Delta$.
For any $\lambda>0$, let $\Sigma_{\lambda}=h_{\lambda}(\Sigma)$, where
$h_{\lambda}$ be the isometry of $\B^3$ defined
by \eqref{eq:hyperbolic_translation},
then $\{\Sigma_{\lambda}\}_{0<\lambda<\infty}$ foliates $\B^3$ by
Proposition \ref{prop:star_shaped_Jordan_curve}.
Since $\partial{}D=\alpha\subset\Sigma$ by assumption,
there exists some $\lambda_{0}$ such that $\Sigma_{\lambda_{0}}$ is tangent to
$D$ from one side, which is impossible because of the maximum principle.
Therefore any curve of $\Sigma\cap\CC$ must be essential in $\CC$.
\end{proof}

\noindent\textbf{Claim 2}: $\Sigma$ \emph{and} $\CC$ \emph{intersect transversely}.

\begin{proof}[{\bf Proof of Claim 2}]
Otherwise we may assume that $\Sigma$ and $\CC$ intersect at a point $p$
non-transversely.
By \cite[Lemm 2]{MY1982(t)} there exist neighborhoods $U\subset\CC$ and
$V\subset\Sigma$ of $p\in\Sigma\cap\CC$ such that $U$ and $V$ intersect
along a finite number of curves passing through $p$ and the intersection
is transversal at points other than $p$
(see also \cite[Figure 1.2 and Lemma 1.4]{FHS83}).

By applying an isometry of $\B^3$, we may assume that $p$ is the origin of
$\B^3$ and the unit normal vector to both $\CC_{a}$ and $\Sigma$
at $p$ is parallel to the $w$-axis
(see the second paragraph in $\S$\ref{sec:introduction} for the definition
of $\B^3$). 
Let $g_{t}$ be a translation along the $w$-axis about distance $t$ for
$t\in(-\infty,\infty)$, then $g_{t}$ is an isometry of
$\B^{3}$. Let $\varepsilon>0$ be a sufficiently small number such that
$g_{t}(\partial_{\infty}\CC_{a})\cap\Gamma=\emptyset$ if $|t|<\varepsilon$.
According to \cite[Lemma 1.5]{FHS83}, we may slightly translate $\CC_{a}$
along the $w$-axis via $g_{t}$ for $|t|\ll\varepsilon$ so that
the minimal disk $\Sigma$ intersects the catenoid $g_{t}(\CC_{a})$
transversely at a finite number of simple closed curves
(see \cite[Figure 1.2]{FHS83}),
which are all null homotopic in $g_{t}(\CC_{a})$ by topological arguments.
But this is impossible according to Claim 1,
so $\Sigma$ must intersect $\CC_{a}$ transversely, and Claim 2 is proved.
\end{proof}

\noindent\textbf{Claim 3}: $\Sigma\cap\CC$ \emph{consists of exactly one simple
closed curve that is essential in the spherical catenoid} $\CC$.

\begin{proof}[{\bf Proof of Claim 3}]
If $\Sigma$ intersects $\CC$ more than once, then there is a compact portion
$\CC'$ of $\CC$ such that $\partial\CC'\subset\Sigma$ consists of two
components in $\Sigma\cap\CC$ and $\CC'$ is totally contained in one component
of $\B^{3}\setminus\Sigma$. Since $\{\Sigma_{\lambda}\}_{0<\lambda<\infty}$
foliates $\B^3$ by Proposition \ref{prop:star_shaped_Jordan_curve}, there
exists some $0<\lambda_{0}<1$ or $\lambda_{0}>1$ such that $\Sigma_{\lambda_{0}}$
is tangent to $\CC'$ from one side.
As usual, this is impossible because of the maximum principle.
This implies that $\Sigma$ intersects $\X$ exactly once, and so
$\Sigma$ intersects $\CC$ exactly once.
\end{proof}

Now each component of $\Sigma\cap\CC$ is a simple closed curve,
which is essential in $\CC$.
Let $\alpha$ denote one of the components of $\Sigma\cap\CC$.
Since $\X$ is a subregion of $\H^3$ with mean convex boundary (actually
$\partial\X=\CC$) and $\alpha$ is null homotopic in $\X$, the curve $\alpha$
must bound an area minimizing disk $\Delta$ in $\X$
by \cite{AS79,MY1982(t),MY1982(mz)}, which is also a subdomain of $\Sigma$.
\end{proof}

\subsection{Existence of area minimizing annuli}
\label{subsec:proof_of_main_theorem_II}

Let $\CC$ be any \emph{area minimizing} catenoid in $\B^{3}$ (i.e., the distance
from $\CC$ to its rotation axis is $\geq{}a_{L}$) asymptotic to disjoint round circles
$C_{1}$ and $C_{2}$ in $S_{\infty}^{2}$.
Suppose that $\Gamma_{1}$ and $\Gamma_{2}$ are disjoint \emph{star-shaped}
Jordan curves contained in the annulus-type component of
$S_{\infty}^{2}\setminus(C_{1}\cup{}C_{2})$ such that
$d(\Gamma_{1},\Gamma_{2})<2\varrho(a_{L})$, and that
$\Sigma_{1}$ and $\Sigma_{2}$ are area minimizing disks asymptotic to
$\Gamma_{1}$ and $\Gamma_{2}$ respectively.
By Proposition \ref{lem:catenoid_intersection}, $\alpha_{i}=\CC\cap\Sigma_{i}$ is a Jordan curve
in $\B^3$ for $i=1,2$.

Since $\Gamma_{1}$ and $\Gamma_{2}$ are disjoint,
it's well known that the disk-type area minimizing surfaces $\Sigma_{1}$
and $\Sigma_{2}$ are also disjoint (see for example \cite[Lemma 1.2]{FHS83}).
We need some notations:
\begin{itemize}
  \item Let $\Bcal$ be the subregion of $\B^3$ such that
        $\partial\Bcal=\Sigma_{1}\cup\Sigma_{2}$ and
        $\partial_{\infty}\Bcal$ is the annulus-type component of
        $S_{\infty}^{2}\setminus(\Gamma_{1}\cup\Gamma_{2})$, then $\Bcal$ is a
        subregion of $\B^3$ with mean convex boundary.
  \item Let $\CC'=\CC\cap\Bcal$, then $\CC'$ is a compact annulus-type minimal
        surface with $\partial\CC'=\alpha_{1}\cup\alpha_{2}$, where
        $\alpha_{i}=\CC\cap\Sigma_{i}$ for $i=1,2$.
  \item Suppose that $\Delta_{i}$ is the disk-type subdomain of $\Sigma_{i}$
        such that $\partial\Delta_{i}=\alpha_{i}$ for $i=1,2$.
\end{itemize}

\begin{lemma}\label{lem:existence_minimal_annulus}
Using the above settings.
Suppose $\gamma_{i}\subset\Sigma_{i}\setminus\overline{\Delta}_{i}$
is a rectifiable Jordan curve for $i=1,2$, then there exists an embedded compact
annulus-type area minimizing
surface $\Pi\subset\B^3$ such that $\partial\Pi=\gamma_{1}\cup\gamma_{2}$.
\end{lemma}

\begin{proof}
Let $\Sigma_{i}'$ be the compact disk-type subdomain of $\Sigma_{i}$
such that $\partial{}\Sigma_{i}'=\gamma_{i}$ for $i=1,2$. Obviously
$\Sigma_{i}'$ is the area minimizing disk spanning
$\gamma_{i}$ for $i=1,2$. By Proposition \ref{lem:catenoid_intersection},
we can define an embedded compact annuls $S\subset\Bcal$ whose boundary
is $\gamma_{1}\cup\gamma_{2}$:
\begin{equation}
   S=\overline{\CC'\cup(\Sigma_{1}'\setminus\Delta_{1})
   \cup(\Sigma_{2}'\setminus\Delta_{2})}\ .
\end{equation}
Since $\CC$ is assumed to be an area minimizing catenoid, we have
the inequality $\area(\CC')<\area(\Delta_{1})+\area(\Delta_{2})$.
Therefore we have
\begin{equation*}
\begin{aligned}
   \area(S)
      & = \area(\CC')+\area(\Sigma_{1}'\setminus{}\Delta_{1})+
          \area(\Sigma_{2}'\setminus{}\Delta_{2}) \\
      & < \area(\Delta_{1})+\area(\Delta_{2})+
          \area(\Sigma_{1}'\setminus{}\Delta_{1})+
          \area(\Sigma_{2}'\setminus{}\Delta_{2}) \\
      & = \area(\Sigma_{1}')+\area(\Sigma_{2}')\ .
\end{aligned}
\end{equation*}
By \cite[Theorem 7]{AS79} or \cite[Theorem 1]{MY1982(t)},
there exists an area minimizing annulus $\Pi\subset\Bcal$
such that $\partial\Pi=\gamma_{1}\cup\gamma_{2}$.

Next we need show that $\Pi$ is also the area minimizing annulus in $\B^3$.
Otherwise, assume that $\Pi'$ is an area minimizing annulus in $\B^3$
with boundary components $\gamma_{1}$ and $\gamma_{2}$ such that
$\area(\Pi')<\area(\Pi)$.
We shall prove that $\Pi'$ is actually contained in $\Bcal$.
In fact, $\Pi'$ can't intersect the component of $\B^{3}\setminus\Bcal$
bounded by the minimal disk $\Sigma_{1}$ since the
family of the minimal disks $\{h_{\lambda}(\Sigma_{1})\}_{0<\lambda<1}$
foliates this subregion by Proposition \ref{prop:star_shaped_Jordan_curve},
where each $h_{\lambda}$ is defined by \eqref{eq:hyperbolic_translation}.
Similarly, $\Pi'$ can't intersect the component of $\B^{3}\setminus\Bcal$
bounded by $\Sigma_{2}$.

Therefore $\Pi'\subset\Bcal$. But we have proved that $\Pi$ is the
area minimizing annulus in $\Bcal$.
This is a contradiction. So $\Pi$ is the area minimizing
annulus in $\B^3$.
\end{proof}


Now we are able to prove Theorem \ref{thm:main_theorem_II}.

\begin{maintheoremII}
Let $\Gamma_{1}$ and $\Gamma_{2}$ be disjoint rectfiable star-shaped
Jordan curves in $S_{\infty}^{2}$. If the distance between $\Gamma_{1}$ and
$\Gamma_{2}$ is bounded from above as follows
\begin{equation}\tag{\ref{eq:upper_bound}}
   d(\Gamma_{1},\Gamma_{2})<{}2\varrho(a_{L})\approx{}0.876895\ ,
\end{equation}
where $\varrho$ is the function defined by \eqref{eq:Gomes function I}
and $a_{L}\approx{}0.847486$ is the constant given by
Theorem \ref{thm:area_difference},
then there exists an embedded annulus-type area minimizing surface
$\Pi\subset\H^3$, which is asymptotic to $\Gamma_{1}\cup\Gamma_{2}$.

Moreover the upper bound \eqref{eq:upper_bound} is optimal in the following
sense: If there is an area minimizing surface in $\H^3$
asymptotic to two disjoint round circles in $S_{\infty}^{2}$, then the distance
\eqref{eq:distance_circles} between the circles is $\leq{}2\varrho(a_{L})$.
\end{maintheoremII}

\begin{proof}
Because of upper bound \eqref{eq:upper_bound}, there exists an area minimizing
catenoid $\CC$ such that two components of $\partial_{\infty}\CC$ are contained
in the disk-type components of $S_{\infty}^{2}\setminus(\Gamma_{1}\cup\Gamma_{2})$
respectively and $\partial_{\infty}\CC\cap(\Gamma_{1}\cup\Gamma_{2})=\emptyset$.

Let $O$ be a fixed point contained in the region $\Bcal$ of
$\B^{3}\setminus(\Sigma_{1}\cup\Sigma_{2})$ such that
$\partial\Bcal=\Sigma_{1}\cup\Sigma_{2}$ and $\partial_{\infty}\Bcal$
is the annulus component of $S_{\infty}^{2}\setminus(\Gamma_{1}\cup\Gamma_{2})$.
Let $B^{3}(s)\subset\B^{3}$ denote the open $3$-ball of radius $s$
(centered at the origin $O$), that is,
\begin{equation*}
   B^{3}(s)=B^{3}(O,s)=\{x\in\H^{3}\ |\ \dist(x,O)<s\}\ .
\end{equation*}
There exists a sufficiently large positive number $r_{0}$ such that the following
two conditions are satisfied for any $s>r_{0}$ 
(see Proposition \ref{prop:single_disk}):
\begin{itemize}
  \item $B^{3}(s)\cap\Sigma_{i}$ consists of exactly one disk, denoted by
        $\Sigma_{i}(s)$ for $i=1,2$.
  \item $B^{3}(s)\cap\Ascr$ is homeomorphic to $\Ascr$, where $\Ascr$ is
        the annulus-type part of the boundary
        of $\CH(\Gamma_{1}\cup\Gamma_{2})$, which is asymptotic to
        $\Gamma_{1}\cup\Gamma_{2}$.
\end{itemize}
For $i=1,2$, we define a Jordan curve $\gamma_{i}(s)$ as follows
\begin{equation*}
   \gamma_{i}(s)=\partial{}\Sigma_{i}(s)=\Sigma_{i}\cap\partial{}B^{3}(s)\ .
\end{equation*}
Because of Proposition \ref{prop:single_disk}, $\gamma_{i}(s)\to\Gamma_{i}$
as $s\to\infty$ for $i=1,2$.

According to Lemma \ref{lem:existence_minimal_annulus}, there exists
an embedded compact annulus-type area minimizing surface $\Pi(s)\subset{}B^{3}(s)$ with
$\partial\Pi(s)=\gamma_{1}(s)\cup\gamma_{2}(s)\subset\partial{}B^{3}(s)$ for each
$s\geq{}r_{0}$ so that
\begin{itemize}
  \item $\area(\Pi(s))<\area(\Sigma_{1}(s))+\area(\Sigma_{2}(s))$, and
  \item $\Pi(s)$ is contained in $\CH(\Gamma_{1}\cup\Gamma_{2})$, $\overline{B^{3}(s)}$
        and $\overline{\Bcal}$, where $\Bcal$ is the subregion of $\B^3$ bounded
        by $\Sigma_{1}$ and $\Sigma_{2}$.
\end{itemize}

\begin{claim}For any $0<r<s$, there exists a constant $C_{r}$, depending
only on $r$, $\Sigma_{1}$ and $\Sigma_{2}$ such that
$\area(\Pi(s)\cap{}B^{3}(r))\leq{}C_{r}$.
\end{claim}

\begin{proof}Let $\theta=\theta_{1}+\theta_{2}$, where $\theta_{i}$ is
the density at infinity of $\Sigma_i$ for $i=1,2$.
According to Theorem \ref{thm:finite_density_at_infinity}, both
$\theta_1$ and $\theta_2$ are finite, so is $\theta$.

For any $0<r\leq{}s$, we have
\begin{equation*}
\begin{aligned}
   \frac{\area(\Pi(s)\cap{}B^{3}(r))}{4\pi\sinh^{2}(r/2)}
       &\leq\frac{\area(\Pi(s))}{4\pi\sinh^{2}(s/2)}\\
       &\leq\frac{\area(\Sigma_{1}(s))+\area(\Sigma_{2}(s))}{4\pi\sinh^{2}(s/2)}\\
       &=\frac{\area(\Sigma_{1}(s))}{4\pi\sinh^{2}(s/2)}+
         \frac{\area(\Sigma_{2}(s))}{4\pi\sinh^{2}(s/2)}\\
       &\leq\theta_{1}+\theta_{2}=\theta\ ,
\end{aligned}
\end{equation*}
where we use the facts $\Pi(s)\cap{}B^{3}(s)=\Pi(s)$ and
$\Sigma_{i}(s)=\Sigma_{i}\cap{}B^{3}(s)$ for $i=1,2$, therefore
\begin{equation}\label{eq:upperbound}
   \area(\Pi(s)\cap{}B^{3}(r))\leq\theta\cdot{}4\pi\sinh^{2}(r/2)=:C_{r}
\end{equation}
for all $0<r\leq{}s$. The proof of the Claim is complete.
\end{proof}

Pick up a sequence of increasing positive real numbers $r_{0}<r_{1}<r_{2}<\cdots$
such that $r_{k}\to\infty$ as $k\to\infty$.
According to Lemma \ref{lem:existence_minimal_annulus},
there exists an area minimizing annulus $\Pi(r_k)\subset\B^3$
spanning $\gamma_{1}(r_k)$ and $\gamma_{2}(r_k)$.
Now we have a sequence of compact annulus-type area minimizing surfaces
$\Pi(r_1),\Pi(r_2),\ldots,\Pi(r_k),\dots$ so that 
$\area(\Pi(r_k))<\area(\Sigma_{1}(r_{k}))+\area(\Sigma_{2}(r_{k}))$
for $k=1,2,\ldots$.

We shall prove that $\{\Pi_{k}:=\Pi(r_k)\}_{k\geq{}1}$ converges smoothly to
a complete area minimizing annulus $\Pi\subset\B^3$ which is asymptotic to
$\Gamma_{1}\cup\Gamma_{2}$ as $k\to\infty$. This can be
done via geometric measure theory, the reader can check \cite{AS79,And83}
for details. Here we just sketch the whole process:

1) For $k=1,2,\ldots$, we can associate each (compact) area minimizing annulus
$\Pi_k$ with a varifold $\v(\Pi_k)\in{}V_{2}(\H^3)$ (see \cite[$\S$3.5]{All75}).

2) According to the estimate \eqref{eq:upperbound} in the above claim, we have
\begin{equation*}
   V=\lim_{k\to\infty}\v(\Pi_k)\in{}V_{2}(\H^3)\
\end{equation*}
by the weak convergence of Randon measures (see \cite[Theorem 3.2]{CM11}).

3) By Theorem 2 in \cite{AS79} or Proposition 3.5 in \cite{CM11},
the varifold $V$ is stationary.

4) According to Allard's regularity theorem in \cite[$\S$8]{All75} and the arguments
in sections 4, 5 and 6 of \cite{AS79}, for each point
$x_{0}\in\spt{}\|V\|$ there is a positive integer $n_{x_0}$, a $\rho_{x_0}>0$,
and an analytic minimal surface $\Sigma_{x_0}$ such that
\begin{equation*}
   V{\niv}{}B^{3}(x_{0},\rho_{x_0})\times{}G(3,2)=n_{x_0}\v(\Sigma_{x_0})\ .
\end{equation*}
This implies that there exists a smooth minimal surface $\Pi\subset\H^{3}$ with
$\partial_{\infty}\Pi=\Gamma_{1}\cup\Gamma_{2}$ such that $V=\v(\Pi)$.
By the arguments in section 9 of \cite{AS79}, this minimal surface $\Pi$
is of annulus-type.

5) It's well known that a limit of area minimizing surfaces is itself area
minimizing. The proof can be found in the last paragraph of the proof of
Theorem 3.1 in \cite{MY92}.
Therefore the minimal surface $\Pi$ is (homotopically) area minimizing, that is,
any compact subdomain of $\Pi$ is an area minimizing surface.

To show that the upper bound \eqref{eq:upper_bound} is optimal,
let's consider the special case
when $\Gamma_{1}$ and $\Gamma_{2}$ are two round circles in $S_{\infty}^{2}$.
If $\Sigma\subset\H^3$ is an area minimizing surface
asymptotic to $\Gamma_{1}\cup\Gamma_{2}$, then
$d(\Gamma_{1},\Gamma_{2})\leq{}2\varrho(a_{L})$
by Theorem \ref{thm:main_theorem_I_prime}.

Now the proof of Theorem \ref{thm:main_theorem_II} is done.
\end{proof}


\bibliographystyle{amsalpha}
\bibliography{ref_annulus}
\end{document}